\documentclass[11pt,reqno]{amsart}
\usepackage{amsmath}
\usepackage{amsfonts,amsthm,amssymb}
\usepackage{mathabx}
\usepackage{bm}
\usepackage{amscd}
\usepackage[latin2]{inputenc}
\usepackage{t1enc}
\usepackage[mathscr]{eucal}
\usepackage{indentfirst}
\usepackage{graphicx}
\usepackage{graphics}
\usepackage{pict2e}
\usepackage{epic}
\numberwithin{equation}{section}
\usepackage[margin=3cm]{geometry}
\usepackage{epstopdf} 
\usepackage[colorlinks,linkcolor=blue]{hyperref}
\usepackage[capitalise,noabbrev]{cleveref}
\usepackage{amscd}
\usepackage{latexsym}
\usepackage{amssymb}
\usepackage{amsmath}
\usepackage{amsthm}
\usepackage{indentfirst}
\usepackage{pifont}
\usepackage{graphicx}
\usepackage{geometry}
\usepackage{subfig}
\usepackage{cases}
\usepackage{hyperref}
\usepackage{cite}
\usepackage{color}

\usepackage{todonotes}
\crefformat{equation}{(#2#1#3)}
\crefmultiformat{equation}{(#2#1#3)}{ and~(#2#1#3)}{, (#2#1#3)}{ and~(#2#1#3)}
\crefrangeformat{equation}{(#3#1#4) to~(#5#2#6)}
\usepackage[normalem]{ulem}

\allowdisplaybreaks

\theoremstyle{plain}
\newtheorem{Thm}{Theorem}[section]
\newtheorem*{Thm*}{Theorem}
\newtheorem{Lem}[Thm]{Lemma}

\newtheorem{Prop}[Thm]{Proposition}

\theoremstyle{definition}

\newtheorem{Rem}[Thm]{Remark}
\newtheorem{?}[Thm]{Problem}

\newcommand{\R}{\mathbb{R}}

\newcommand{\di}{\displaystyle}

\newcommand{\norm}[1]{\left\lVert#1\right\rVert}

\begin{document}
	
		\title[Stability of 2D non-self-similar rarefaction wave]{Nonlinear asymptotic stability of non-self-similar rarefaction wave for two-dimensional viscous Burgers equation}
		
		
\author[Huang]{Feimin Huang}
\address[Feimin Huang]{\newline Institute of Applied Mathematics, AMSS, Chinese Academy of Sciences, Beijing 100190, P. R. China
and School of Mathematical Sciences, University of Chinese Academy of Sciences,
\newline Beijing 100049, P. R. China}
\email{fhuang@amt.ac.cn}

\author[Qiu]{Guiqin Qiu}
\address[Guiqin Qiu]{\newline  Institute of Applied mathematics, AMSS, Chinese Academy of Sciences, Beijing 100190, P.R. China}
\email{guiqinqiu@126.com}

\author[Wang]{Yi Wang}
\address[Yi Wang]{\newline Institute of Applied Mathematics, AMSS, Chinese Academy of Sciences, Beijing 100190, P. R. China
\newline
and School of Mathematical Sciences, University of Chinese Academy of Sciences,
\newline Beijing 100049, P. R. China}
\email{wangyi@amss.ac.cn}

\author[Yang]{Xiaozhou Yang}
\address[Xiaozhou Yang]{\newline Wuhan Institute of Physics and Mathematics, Innovation Academy for Precision Measurement Science and Technology, 
	Wuhan 430071, P.R. China
	\newline
and School of Mathematical Sciences, University of Chinese Academy of Sciences,
\newline Beijing 100049, P. R. China}
\email{xzyang@wipm.ac.cn}


\maketitle

\vspace{-0.6cm}

\begin{abstract}
We investigate the large time behavior of solutions to the two-dimensional viscous Burgers equation $u_t+uu_x+uu_y=\Delta u$, toward a non-self-similar rarefaction wave of inviscid Burgers equation with two initial  constant states, seperated by a curve $y=\varphi(x)$, and prove that the above 2D non-self-similar rarefaction wave is time-asymptotically stable. Furthermore, we can get the decay rate. Both the rarefaction wave strength and the initial perturbation can be large.

The main difficulty comes from the fact that the initial discontinuity of 2D non-self-similar rarefaction is a curve $y=\varphi(x)$, while the initial discontinuity of self-similar rarefaction is either a point in 1D or a straight line in 2D. Fortunately, we uncover a novel property that the  non-self-similar inviscid rarefaction wave is also asymptotically stable with respect  to the discontinuity curve $y=\varphi(x)$. More precisely, let $u_i^R(x,y,t), i=1,2$ be  the corresponding non-self-similar rarefaction with the initial discontinuity curve $y=\varphi_i(x)$, then $\|u_1^R-u_2^R\|_{L^\infty}\le \frac{C}{t}$ if $\|\varphi_1(x)-\varphi_2(x)\|_{L^\infty}$ is bounded. Based on this property, any curve $y=\varphi(x)$ can be classified as a standard  polyline  \begin{equation*}
L(x)=
\left\{\begin{array}{ll}
	\di  k_-x+c,&\di x< 0,\\[1mm]
	\di k_+x+c,&\di x\geq 0,
\end{array}
\right.
\end{equation*}
where $\lim_{x\to \pm\infty}\frac{\varphi(x)}{x}=k_\pm$. 

Then we construct the approximate smooth rarefaction wave of the viscous Burgers equations based on the above polyline, and 
convert this wave to the self-similar planar rarefaction one through a new nonlinear coordinate transformation with the price that the 2D viscous Burgers equation is changed into a parabolic equation with variable and mixed derivative viscosities.  Another advantage is that the main error terms generated by the new approximate smooth rarefaction in the new parabolic equation are integrable in $\R^2$.
This approach enables us to overcome the main difficulty mentioned above.
Finally, the time-asymptotic stability analysis is carried out to the viscous rarefaction wave and the transformed 2D viscous Burgers equation by using suitable time-weighted $L^p$-energy estimates, which gives a first result on the nonlinear time-asymptotic stability of 2D non-self-similar rarefaction wave.
\vspace{.2cm}

\vspace{.2cm}

\end{abstract}

\vspace{.2cm}




\section{Introduction and Main Result}

We investigate the time-asymptotic stability of 2D non-self-similar  rarefaction wave for 2D viscous Burgers equation
 \begin{equation}\label{eq}
 u_t+uu_x+uu_y=u_{xx}+u_{yy},
 \end{equation}
with the initial data
 \begin{equation}\label{idata}
 u(0,x,y)=u_0(x,y)
 \end{equation}
 satisfying 
 \begin{equation}\label{idata1}
 u_0(x,y)\rightarrow u_{\pm}\quad \text{as} \quad Z(x,y)\rightarrow \pm\infty,
 \end{equation}
where $u_{\pm}$ are prescribed constants satisfying $u_-<u_+$, and $Z(x,y)$ is an implicit function determined by 
\begin{equation}\label{imz}
 y-Z-\varphi (x-Z)=0,
 \end{equation}
here $y=\varphi(x)$ is the initial discontinuity curve of 2D non-self-similar rarefaction wave with $\varphi(x)\in C^{2}(\mathbb R)$ satisfying 
\begin{equation}\label{2DH}1-\varphi'(x)>d_0>0\end{equation}
for all $x\in \mathbb R$, and $d_0$ being some positive constant. Note that the condition \eqref{2DH} guarantee the global existence and uniqueness of the implicit function $Z(x,y)$ in \eqref{imz}.
In particular, if $\varphi(x)=kx+c~(k<1)$ is a straight line with the slope $k<1$, then $$Z(x,y)= \frac{y-kx-c}{1-k},$$ which is exactly the classical planar self-similar rarefaction wave case up to a linear transformation. Therefore, the main result and the analysis in the sequel can be fully applied to the planar self-similar rarefaction wave case.

As the time $t\rightarrow +\infty$, we expect that the large-time behavior of the solution to \eqref{eq}-\eqref{idata} is determined by the following 2D non-self-similar Riemann problem

\begin{equation}\label{eq0}
	u_t+uu_x+uu_y=0,
\end{equation}
\begin{equation}\label{0data}
	u(0,x,y)=\begin{cases}
		u_-,~y-\varphi (x)<0,\\
		u_+,~y-\varphi (x)>0,\end{cases}
\end{equation}
 where $u_-<u_+$ and $y-\varphi(x)=0$ is the general initial discontinuity curve satisfying \eqref{2DH}. Then the unique entropic solution to \eqref{eq0}-\eqref{0data} is 2D non-self-similar rarefaction wave, which is already constructed in \cite{Yang1999}, with the explicit formula
\begin{equation}\label{Rs}
	u^R(t,x,y)=\begin{cases}
		u_-,&y< u_-t+\varphi (x-u_-t),\\[2mm]
		\frac{Z(x,y)}{t},&u_-t+\varphi (x-u_-t)\leq y\leq u_+t+\varphi (x-u_+t), \\[1mm]
		u_+,&y>u_+t+\varphi (x-u_+t),
	\end{cases}
\end{equation}
where $Z(x,y)$ in  the intermediate state  is the implicit function determined by 
\eqref{imz}. 
The 2D non-self-similar rarefaction wave $u^R(t,x,y)$ in \eqref{Rs} can be illustrated by \cref{f2}.

\begin{figure}[htbp]
 	\centering
	\scalebox{1}{\includegraphics[scale=0.49,trim={0cm 0cm 0cm 0cm},clip]{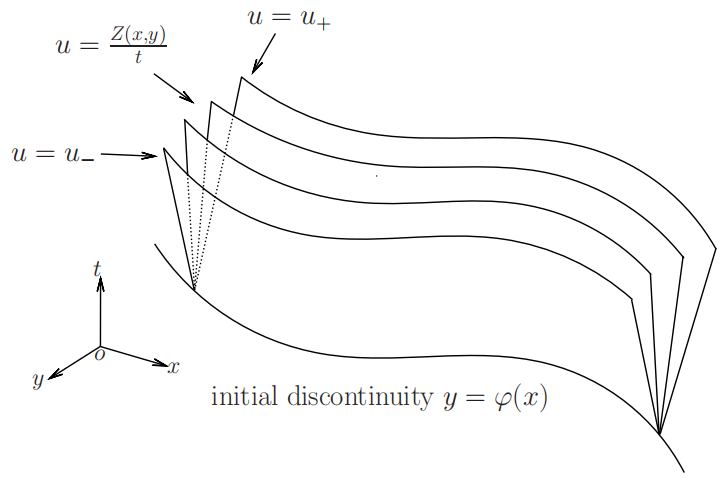}}\\
	\caption{2D non-self-similar rarefaction wave
 	}\label{f2}
 \end{figure}
Combining \eqref{imz} and \eqref{2DH}, 2D non-self-similar rarefaction wave \eqref{Rs} can equivalently be expressed as follows:
\begin{equation}\label{Rs1}
	u^R(t,x,y)=\begin{cases}
		u_-,&Z(x,y)<u_-t,\\[2mm]
		\frac{Z(x,y)}{t},&u_-t\leq Z(x,y)\leq u_+t,\\[1mm]
		u_+,&Z(x,y)>u_+t.
	\end{cases}
\end{equation}

The main purpose of the present paper is to prove the time-asymptotic stability of the above 2D non-self-similar  rarefaction wave $u^R(t,x,y)$ to 2D viscous Burgers equation \eqref{eq}-\eqref{idata}. 
 
There are plenty of literatures on the time-asymptotic stability of 1D self-similar rarefaction wave to 1D viscous conservation laws since the classical work of Il'in and Oleinik \cite{Ilin1960} in $1960$, see \cite{H1988,HN1991,KT2004,MN1986} and the references therein for details. 
Recently, there are also many works on the time-asymptotic stability of 1D self-similar planar rarefaction wave for multi-dimensional viscous conservation laws.  
Xin \cite{Xin1990} first proved the time-asymptotic stability of 1D planar rarefaction waves to multi-dimensional scalar viscous conservation laws by the fundamental $L^2$ energy methods. Then Ito \cite{Ito1996} and Nishikawa and Nishihara \cite{NN2000} showed the corresponding time decay rate toward the  planar rarefaction wave in \cite{Xin1990}, and Hokari and Matsumura \cite{HoM} generalized \cite{Xin1990} to a special 2D system with artificial viscosities.  Chen and Zhu  \cite{CZ2010} obtained the time decay rate of strong planar rarefaction wave to multi-dimensional scalar viscous conservation laws with degenerate viscosity. Kawashima, Nishibata and Nishikawa \cite{KNN2004} established the general $L^p$ stability framework for planar rarefaction wave  to multi-dimensional scalar viscous conservation laws. Very recently, the time-asymptotic stability and/or vanishing dissipation limit of self-similar planar rarefaction wave to multi-dimensional scalar conservation laws and 2D/3D compressible Navier-Stokes equations with physical viscosities were also well-established in \cite{HY2021, LWW1, LWW2, LWW, LW}  when the transverse directions are periodic. Note that all the above results are concerned with the stability of 1D self-similar rarefaction wave to 1D/multi-dimensional viscous conservation laws. As far as we know, there are no any result on the time-asymptotic stability of non-self-similar multi-dimensional wave patterns to scalar viscous conservation laws \eqref{eq}.
The main purpose of the present paper is to show the time-asymptotic stability of 2D non-self-similar rarefaction wave $u^R(t,x,y)$ in \eqref{Rs} for the 2D viscous Burgers equation. 

First,  since 2D non-self-similar rarefaction wave \eqref{Rs} is only Lipschitz continuous, to prove its time-asymptotic stability to the viscous Burgers equation \eqref{eq} with second order derivatives,  we need  to construct its suitable smooth approximation ansatz. 
Since 2D non-self-similar rarefaction wave \eqref{Rs} has the form $\di u^R\big(\frac{Z(x,y)}{t}\big)$ as the media state, and from \cref{LZ} the function $x-Z(x,y)$ can be expressed as $G(x-y)$ for some implicit function $G$ determined by $-(x-y)+G-\varphi (G)=0$ and satisfying $G'(\xi)=\frac{1}{1-\varphi'(G(\xi))}>0$.
Therefore, we can introduce the following spatial coordinate transformation 
\begin{equation}\label{nct}
\xi=x-y,~\eta=Z(x,y),
\end{equation}  
 to convert the 2D non-self-similar rarefaction wave  \eqref{Rs} or \eqref{Rs1} to 1D self-similar planar rarefaction wave $u^R(\frac{\eta}{t})$ in $(\xi, \eta)$-coordinate. Motivated by 1D stability, we can construct the smooth ansatz $w(t,\eta)$ as
  \begin{equation}\label{www1}
 \left\{\begin{array}{ll}
\di  	w_t+ww_{\eta}=0,\\
\di	w(0,\eta)=w_0(\eta):=\frac{u_++u_-}{2}+\frac{u_+-u_-}{2}\kappa \int_{0}^{\eta}\frac{1}{1+\alpha^2}d\alpha,
\end{array}
\right. \end{equation}
where the positive constant
\begin{equation}\label{kapa}
	\kappa :=\left(\int_{0}^{+\infty}\frac{1}{1+\alpha^2}d\alpha\right)^{-1}.
\end{equation}
Since $w_0^\prime(\eta)>0,$ 1D inviscid Burgers equation \eqref{www1} has a unique global classical solution $w(t,\eta)$ satisfying $w_\eta(t,\eta)>0$ for all $t\in \mathbb{R}^+$.

Correspondingly, under the new spatial coordinate $(\xi, \eta)$, 2D viscous Burgers equation \eqref{eq} is transformed into
\begin{equation}\label{nlte}
	u_t+uu_{\eta} =K(\xi)u_{\eta\eta}+A(\xi)u_{\xi\eta}+2u_{\xi\xi}+B(\xi)u_{\eta},
\end{equation}
where 
\begin{equation}\label{coeff}
	K(\xi)=\frac{1+\varphi'^2(G(\xi))}{(1-\varphi'(G(\xi)))^2}>0,\quad
	A(\xi)=\frac{-2(1+\varphi'(G(\xi)))}{1-\varphi'(G(\xi))},\quad
	B(\xi)=\frac{-2\varphi''(G(\xi))}{(1-\varphi'(G(\xi)))^3},
\end{equation}
and detailed derivations can be found in  \cref{LZ} below. We emphasize that some new difficulties appear due to the dissipation terms with variable coefficients and the mixed derivitives on the right-hand side of \eqref{nlte}. On the other hand, all the coefficients $K, A$ and $B$ in \eqref{coeff} depend only on the variable $\xi$, which is quite important in our later analysis.

Second, if we carry out the $L^p$-stability analysis around the inviscid approximate rarefaction wave profile  $w(t,\eta)$ in \eqref{www1} for 2D transformed viscous Burgers equation \eqref{nlte} in $\mathbb{R}^2$, then the error terms due to the inviscid profile to the viscous equation are not $L^p$-integrable on $\mathbb {R}^2$ with respect to the spatial variable. This difficulty is similar to the one occurred in the time-asymptotic stability of 1D self-similar planar rarefaction wave to the multi-dimensional scalar viscous conservation laws, but the main difference here is the dissipation terms on the right-hand side of \eqref{nlte} are much more complicated and include the variable coefficients and the mixed derivatives. Note that the last term $B(\xi)u_{\eta}$ on the right-hand side of \eqref{nlte} is a first-order derivative  term with the coefficient $B(\xi)$ having second-order derivative of $\varphi$ with respect to $G(\xi)$, but it really comes from the original dissipation terms of the equation \eqref{eq} and can also be viewed as the dissipation term  under the new spatial coordinate $(\xi, \eta)$. Inspired by the time-asymptotic stability of 1D self-similar planar rarefaction wave to the multi-dimensional scalar viscous conservation laws in Xin \cite{Xin1990},  we first solve  the following
1D viscous Burgers equation for the positive variable viscosity coefficient $K(\xi)$
\begin{equation}\label{varf}
 \left\{\begin{array}{ll}
\di  
v_t+vv_{\eta}=K(\xi)v_{\eta\eta},\\[2mm]
v(0,\xi, \eta)=w_0(\eta),
\end{array}
\right.
\end{equation}
around the smooth rarefaction ansatz $w(t,\eta)$ in \eqref{www1} with the same initial value $w_0(\eta)$. Due to the  variable viscosity coefficient $K(\xi)$, the solution to the equation \eqref{varf} also depends on the variable $\xi$, denoted by $v(t,\xi,\eta)$.  By maximum principle, we have the solution $v(t,\xi,\eta)$ still satisfies $v_\eta(t,\xi,\eta)>0$ for all $t\in \mathbb{R^+}$ and $\xi\in \mathbb{R}$. The other detailed properties of the solution $v(t,\xi,\eta)$ will be listed in Subsection 3.1. Therefore, we can overcome the difficulty due to the error term with respect to the viscosity term $K(\xi)u_{\eta\eta}$.



Next, we deal with the remaining error terms due to $A(\xi)u_{\xi\eta}+2u_{\xi\xi}+B(\xi)u_{\eta}$. Since the viscous rarefaction wave profile $v(t,\xi,\eta)$ in \eqref{varf} depends on the variable $\xi$ due to the non-self-similar structure, which is quite different from the self-similar planar rarefaction wave case in \cite{Xin1990}, we need to further assume that the initial discontinuity $y=\varphi(x)\in C^2(\mathbb{R})$ is linear at infinity, and can be an arbitrary $C^2$ function in any finite domain,
that is, there exist constants $k_1,k_2,c_1,c_2\in \mathbb R$ such that $\varphi(x)\in C^2(\mathbb{R})$ and
\begin{equation}\label{cl}
\norm{\varphi(x)-(k_1x+c_1)}_{ H^1(-\infty,0)}<+\infty,~\norm{\varphi(x)-(k_2x+c_2)}_{H^1(0,+\infty)}<+\infty.
\end{equation} 
If the initial discontinuity $y=\varphi(x)=kx+c~(k<1)$ corresponds to the self-similar planar rarefaction wave, then the assumption \eqref{cl} obviously holds true by choosing $k_1=k_2=k,~c_1=c_2=c$. 
 
Under the assumption \eqref{cl}, we can reformulate our problem by constructing a modified poly-line $y=L_{\epsilon_0}(x)\in C^{\infty}(\mathbb{R})$ satisfying
\begin{equation}\label{Le0}
 L_{\epsilon_0}'(x)=
 \left\{\begin{array}{ll}
 \di  k_1,&\di x\leq -\epsilon_0,\\[1mm]
 \di k_2,&\di x\geq \epsilon_0,
 \end{array}
 \right.
\end{equation}
for some fixed positive constant $\epsilon_0$.
See \cref{f1}.  
 \begin{figure}[htbp]
	\centering
	\scalebox{1}{\includegraphics[scale=0.8,trim={0cm 0cm 0cm 0cm},clip]{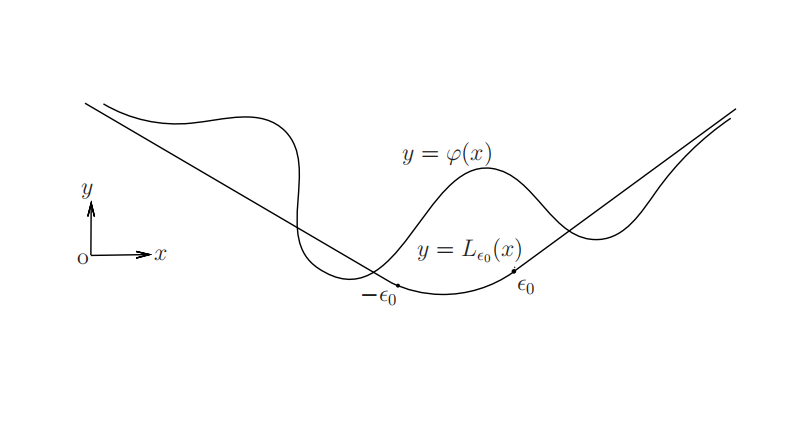}}
	\caption{Initial discontinuity curve $\varphi(x)$ and modified poly-line $L_{\epsilon_0}(x)$ 
	}\label{f1}
\end{figure}
We claim that the 2D non-self-rarefaction wave for the initial discontinuity $y=\varphi(x)$ satisfying \cref{2DH} and \cref{cl} is time-asymptotically equivalent to the one for the the initial discontinuity $y=L_{\epsilon_0}(x)$ defined in \eqref{Le0} for  some positive constant $\epsilon_0$. The detailed proof of this claim will be presented in Section 2.

Since $L_{\epsilon_0}''(x)=0$ for $ |x|\geq \epsilon_0$, and the numerator of the coefficient $B(\xi)$ defined in \eqref{coeff} depends linearly on $L_{\epsilon_0}''$, the spatial support of the error term $B(\xi)v_{\eta}$ lies in the set  $\big\{(\xi, \eta)\in\mathbb{R}^2\big|\xi\in[G^{-1}(-\epsilon_0),G^{-1}(\epsilon_0)]\big\}$. Therefore, the error term due to $B(\xi)u_{\eta}$ is $L^p$-integrable with respect to $(\xi, \eta)\in \mathbb {R}^2$. 
For the error terms concerning $A(\xi)u_{\xi\eta}+2u_{\xi\xi}$, by the chain rule $v_{\xi}=\frac{\partial v}{\partial K}K'(\xi),$ and from \eqref{coeff}, we have
 $$
 K^\prime(\xi)=\frac{2\big(1+L_{\epsilon_0}'(G(\xi))\big)L_{\epsilon_0}''(G(\xi))G^\prime(\xi)}{(1-L_{\epsilon_0}'(G(\xi)))^3}.
 $$
Similarly we can see that $K'(\xi)$,  $A(\xi)v_{\xi\eta}$ and $v_{\xi\xi}$ are completely supported in $\big\{(\xi, \eta)\in\mathbb{R}^2\big|\xi\in[G^{-1}(-\epsilon_0),G^{-1}(\epsilon_0)]\big\}$. Therefore, all the error terms  due to $A(\xi)u_{\xi\eta}+2u_{\xi\xi}+B(\xi)u_{\eta}$ are $L^p$-integrable on $\mathbb {R}^2$ for the initial discontinuity $y=L_{\epsilon_0}(x)$ satisfying \eqref{Le0}.
Then we can use the $L^p$-energy method introduced in \cite{KNN2004} for the perturbation $\phi(t, \xi, \eta):=u(t,\xi, \eta)-v(t,\xi, \eta)$.

However, all the above considerations for the error terms due to the dissipation terms only concern the $L^p$-integrability with respect to $(\xi, \eta)\in \mathbb {R}^2$, and do not involve the integration with respect to the time $t$.
Roughly speaking, carrying out  $L^p$-estimate to the perturbation $\phi:=u-v$, we have
\begin{equation}\label{b0}
	\begin{aligned}
	&\frac{1}{p}\frac{d}{dt}\norm{\phi}_{L^p(R^2)}^p+\frac{p-1}{p}\iint_{\mathbb R^2}v_{\eta}|\phi|^pd\xi d\eta +(p-1)\iint_{\mathbb R^2}|\phi|^{p-2}(K(\xi)\phi_{\eta}^2+A(\xi)\phi_{\xi}\phi_{\eta}+2\phi_{\xi}^2)d\xi d\eta \\&=
	\iint_{\mathbb R^2}A(\xi)v_{\xi\eta}|\phi|^{p-2}\phi d\xi d\eta +2\iint_{\mathbb R^2}v_{\xi\xi}|\phi|^{p-2}\phi d\xi d\eta +\iint_{\mathbb R^2}B(\xi)v_{\eta}|\phi|^{p-2}\phi d\xi d\eta.
\end{aligned}
\end{equation}
If we estimate the last term  on the right-hand side of \eqref{b0} as
\begin{equation}\label{b1}
	\begin{aligned}
\iint_{\mathbb R^2}B(\xi)v_{\eta}|\phi|^{p-2}\phi d\xi d\eta&\leq \epsilon\iint_{\mathbb R^2}v_{\eta}|\phi|^pd\xi d\eta +C_{\epsilon}\iint_{\mathbb R^2}|B(\xi)|v_{\eta}d\xi d\eta\\ 
	&= \epsilon\iint_{\mathbb R^2}v_{\eta}|\phi|^pd\xi d\eta +C_\epsilon,\end{aligned}
\end{equation}
with the small positive constant $\epsilon$ and the positive constant $C_\epsilon$ depending on $\epsilon$. Then we can only have at best
  \begin{equation}\label{b2}
  \norm{\phi}_{L^p}^p\leq C_\epsilon(1+t), 
  \end{equation}
which is not good enough for our time-asymptotic stability. Therefore, we need to use another method to estimate \eqref{b1}.
We observe that 
 \begin{equation}
 	B(\xi)=\frac{-2\varphi''(G(\xi))}{(1-\varphi'(G(\xi)))^3}=\partial_{\xi}\big(\frac{-2}{1-\varphi'(G(\xi))}\big),
 \end{equation}
 and $K(\xi),A(\xi)$ satisfy 
 \begin{equation}\label{para}
 	K(\xi)\alpha_1^2+A(\xi)\alpha_1\alpha_2+2\alpha_2^2\geq d(\alpha_1^2+\alpha_2^2) 
 \end{equation}
for any $\xi\in \mathbb R$ and $(\alpha_1,\alpha_2)\in \mathbb R^2$, where $d$ is a positive constant. Then for $p>2$,  we can deduce from \eqref{b0} after the integration by part that
 \begin{equation}\label{b3}
 	\begin{aligned}
 		&\frac{d}{dt}\norm{\phi}_{L^p(\mathbb R^2)}^p+C_p\iint_{\mathbb R^2}v_{\eta}|\phi|^pd\xi d\eta +C_{p,d}\norm{\nabla(|\phi|^{\frac{p}{2}})}_{L^2(\mathbb R^2)}^2\\\leq&
 		C\iint_{\mathbb R^2}v_{\xi}^2|\phi|^{p-2}d\xi d\eta +C\iint_{\mathbb R^2}v_{\eta}^2|\phi|^{p-2}d\xi d\eta.
 	\end{aligned}
 \end{equation}
Since $\norm{v_{\xi}}_{L^{p}}\leq C(1+t)^{-\frac{1}{2}+\frac{1}{2p}}\ln(1+t)$ and $\norm{ v_{\eta}}_{L^{p}}\leq (1+t)^{-1+\frac{1}{p}}$, we still cannot control the terms on the right-hand side of \eqref{b3} directly. 
Fortunately, using the fact that ${\rm supp}\  v_{\xi}(t,\xi,\eta)\subset \mathbb R_{+} \times[G_1, G_2]\times \mathbb R$ with $G_1:=G^{-1}(-\epsilon_0)$ and $G_2:=G^{-1}(\epsilon_0)$, we have  for any $\xi\in \mathbb{R}$
  \begin{align*}
  	v_{\xi}^2(t,\xi,\eta)=\partial_{\xi}\left(\int_{G_1}^{\xi}v_{\xi}^2(t,\alpha,\eta)d\alpha\right),
  \end{align*}
 and
  \begin{align*}
  	\int_{G_1}^{\xi}v_{\xi}^2(t,\alpha,\eta )d\alpha\leq \int_{G_1}^{G_2}v_{\xi}^2(t,\alpha,\eta)d\alpha.
  \end{align*}
For $p>4$, we have
  \begin{align*}
  	\iint_{\mathbb R^2}v_{\xi}^2|\phi|^{p-2}d\xi d\eta &=-(p-2)\iint_{\mathbb R^2}\left(\int_{G_1}^{\xi}v_{\xi}^2(t,\alpha,\eta)d\alpha\right)|\phi|^{p-4}\phi \phi_{\xi}d\xi d\eta \\
  	&\leq \epsilon \iint_{\mathbb R^2}|\phi|^{p-2}\phi_{\xi}^2d\xi d\eta +C_{\epsilon}\iint_{\mathbb R^2}\left(\int_{G_1}^{G_2}v_{\xi}^2(t,\alpha,\eta)d\alpha\right)^2|\phi|^{p-4}d\xi d\eta \\
  	&\leq \epsilon \norm{\nabla(|\phi|^{\frac{p}{2}})}_{L^2(\mathbb R^2)}^2 +C_{\epsilon}(G_2-G_1)^2\norm{v_{\xi}}_{L^{\infty}}^4\norm{\phi}_{L^{p-4}}^{p-4}\\
	&\leq  \epsilon \norm{\nabla(|\phi|^{\frac{p}{2}})}_{L^2(\mathbb R^2)}^2 +C_{\epsilon}(1+t)^{-2}\ln^4(3+t)\norm{\phi}_{L^{p-4}}^{p-4},
  \end{align*}
  with the small positive constant $\epsilon$ and the positive constant $C_{\epsilon}$ depending on $\epsilon$. 
  
  On the other hand,
$$
  \iint_{\mathbb R^2}v_{\eta}^2|\phi|^{p-2}d\xi d\eta\leq \norm{v_{\eta}}_{L^{\infty}}^2\norm{\phi}_{L^{p-2}}^{p-2}\leq C(1+t)^{-2}\norm{\phi}_{L^{p-2}}^{p-2}.$$
 Therefore, from  \eqref{b3} we have 
\begin{equation}\label{ggg}
	\frac{d}{dt}\norm{\phi}_{L^p(\mathbb R^2)}^p
	\leq
	C(1+t)^{-2}\ln^4(3+t)\norm{\phi}_{L^{p-4}}^{p-4}+C(1+t)^{-2}\norm{\phi}_{L^{p-2}}^{p-2}.
\end{equation}
By combining \eqref{b2} and \eqref{ggg},  we have
$$
\norm{\phi}_{L^{6}(\mathbb R^2)}\leq C\ln(3+t).
$$
Therefore, we can obtain the time decay rate of $\norm{\phi}_{L^p(\mathbb R^2)}$ for $p>6$  through the $L^p$-energy method. Note that the approximate rarefaction wave is time-asymptotically equivalent to its viscous profile, we can finally obtain our main result on the time-asymptotic stability of 2D non-self-similar rarefaction wave \eqref{Rs} stated as 

\begin{Thm}\label{mt}
	Suppose that $u_0(x,y)-w_0(Z(x,y))\in H^1(\mathbb{R}^2)\cap L^\infty(\mathbb{R}^2)$ and the initial discontinuity $y=\varphi(x)$ satisfies \cref{2DH} and \cref{cl}. Then the Cauchy problem \cref{eq}-\cref{idata1} admits a unique global smooth solution $u(t,x,y)$ satisfying

	\begin{equation}\label{mtr}
		\norm{u(t,x,y)-u^R(t,x,y)}_{L^{\infty}(\mathbb{R}^2)} \leq C_{\epsilon}(1+t)^{-\frac{1}{8}+\epsilon}\ln^3(3+t), \qquad \forall t\geq0,
	\end{equation}
	where $u^R(t,x,y)$ is 2D non-self-similar rarefaction wave \eqref{Rs} or \eqref{Rs1},  $\epsilon>0$ is an arbitrarily small constant, and the positive constant $C_{\epsilon}$ is independent of the time $t\geq 0$ but dependent on $\epsilon$.
\end{Thm}
\begin{Rem}
\cref{mt} gives a first result on the nonlinear time-asymptotic stability of 2D non-self-similar rarefaction wave with time decay rare as far as we know, even though the decay rate in \eqref{mtr} may not be optimal.	
\end{Rem}
\begin{Rem}
	If we further assume that the initial data $u_0(x,y)-w_0(Z(x,y))\in L^1(\mathbb R^2)$, then the decay rate in \eqref{mtr} can be improved to $C_{\epsilon}(1+t)^{-\frac{1}{7}+\epsilon}\ln^3(3+t)$.
\end{Rem}
\begin{Rem}
	\cref{mt} can be extended to the more general scalar viscous conservation laws
	\begin{equation}\label{ge}
	u_t+k_1f(u)_x+k_2f(u)_y=\varepsilon\Delta u,
	\end{equation}
	 where
	$k_1,k_2$ are any given nonzero constants, and $f''(u)>\alpha>0$. In fact, we can define the implicit $Z(x,y)$ satisfying 
	$y-k_2Z-\varphi (x-k_1Z)=0$. Note that $x-k_1Z(x,y)$ can be expressed as $G(k_2x-k_1y)$ for some implicit function $G$ determined by $-\frac{1}{k_2}(k_2x-k_1y)+G-\frac{k_1}{k_2}\varphi (G)=0$. 
	Then the non-self-similar rarefaction can be expressed by
	\begin{equation}\label{r2}
		u^R(t,x,y)=\begin{cases}
			u_-,&Z(x,y)<f'(u_-)t,\\[2mm]
			f'^{-1}(\frac{Z(x,y)}{t}),&f'(u_-)t\leq Z(x,y)\leq f'(u_+)t,\\[1mm]
			u_+,&Z(x,y)>f'(u_+)t.
		\end{cases}
	\end{equation}
Under the transformation $$\xi=k_2x-k_1y,~\eta=Z(x,y),$$ 
 we can prove the time-asymptotic stability of the 2D non-self-similar rarefaction wave \eqref{r2} to 2D general scalar viscous conservation laws \eqref{ge}.
\end{Rem}	


\

The rest of the paper is organized as follows.
In Section 2,  we give some preliminary lemmas, and also consider the stability of 2D non-self-similar rarefaction wave solution to the Riemann problem \eqref{eq0} and \eqref{0data}. Then we reformulate the 2D Burgers equation into a transformed parabolic equation with variable and mixed derivative viscosities. In Section 3, we first construct the approximate smooth rarefaction wave and obtain its properties. Then we try to prove the time-asymptotic stability and the time decay rate of approximate rarefaction wave $w(t,\eta)$ to the transformed 2D Burgers equation \eqref{equ} in the new spatial coordinate $(\xi,\eta)$. Finally we  prove the main result, i.e., \cref{mt}. 

\

\section{Preliminaries and  the reformulation of problem}
\vspace{0.3cm}
In this section, we first give some preliminary lemmas, and then consider the stability of 2D non-self-similar rarefaction wave solution to Riemann problem \eqref{eq0} \eqref{0data}  under perturbation to initial discontinuity. Next, we construct a modify poly-line $y-L_{\epsilon_0}(x)=0$, and we point out that in the regime of our main \cref{mt},  the initial discontinuity $y=\varphi(x)$ satisfying \cref{2DH} \cref{cl} is equivalent to the initial discontinuity $y=L_{\epsilon_0}(x)$, this is a critical step. Finally, by introduce an important nonlinear transformation,  we convert the 2D non-self-similar rarefaction wave to a self-similar planar rarefaction wave, and construct its smooth approximation. Thus we shall study the asymptotic stability of non-self-similar rarefaction wave in the new coordinate space.

\subsection{Preliminary Lemmas}
In this subsection, we introduce several interpolation inequalities and some properties of $Z(x,y)$, which are fundamental to our later estimates.
\begin{Lem}[\cite{KNN2004}]\label{esl}
\begin{itemize}
		
		\item[(i).] 
	Let $n\geq 1$, $2\leq p <+\infty$ and $1\leq q\leq p$, then
	\begin{equation}\label{esqn}
		\norm{u}_{L^p(\mathbb{R}^n)}\leq C\norm{\nabla (|u|^{\frac{p}{2}})}_{L^2(\mathbb{R}^n)}^{\frac{2\gamma}{1+\gamma p}}\norm{u}_{L^q(\mathbb{R}^n)}^{\frac{1}{1+\gamma p}},
	\end{equation}
where $\gamma=\frac{n}{2}(\frac{1}{q}-\frac{1}{p})$, and $C$ is a positive constant depending on $p,q$ and $n$. 
	Specially, if $n=1,2$, we have
	\begin{equation}\label{esq1}
		\norm{u}_{L^p(\mathbb{R})}\leq C\norm{\nabla (|u|^{\frac{p}{2}})}_{L^2(\mathbb{R})}^{\frac{2(p-q)}{p(p+q)}}\norm{u}_{L^q(\mathbb{R})}^{\frac{2q}{p+q}},
	\end{equation}
	\begin{equation}\label{esq2}
		\norm{u}_{L^p(\mathbb{R}^2)}\leq C\norm{\nabla (|u|^{\frac{p}{2}})}_{L^2(\mathbb{R}^2)}^{\frac{2(p-q)}{p^2}}\norm{u}_{L^q(\mathbb{R}^2)}^{\frac{q}{p}}.
	\end{equation}
\item[(ii).] 
Let $n\geq 1$ and $2\leq p <+\infty$,  then
	\begin{equation}\label{puin1}
		\norm{\partial_{x_i} u}_{L^p(\mathbb{R}^n)}\leq C\norm{\partial_{x_i} (|\partial_{x_i}u|^{\frac{p}{2}})}_{L^2(\mathbb{R}^n)}^{\frac{2}{p+2 }}\norm{u}_{L^p(\mathbb{R}^n)}^{\frac{2}{p+2}}
	\end{equation}
	for $i=1,...,n$, and $C$ is a positive constant depending only on $p$.
		\end{itemize}
\end{Lem}

The inequalities in the following lemma is very useful for time-weighted energy estimates.
\begin{Lem}\label{est}Let $n\geq 1$, $2\leq p <+\infty$ and $1\leq q\leq p$, denote
	\begin{equation}\label{esj}
		I(i,k,n)=\int_{0}^{t}(1+\tau)^{\alpha-i}\norm{u}_{L^p(\mathbb{R}^n)}^{p-k}(\tau)d\tau,
	\end{equation}
	where $i\geq 0,~0\leq k<p$, and $\alpha>0$ is some sufficiently large number, then for any function $A(t)>0$, it holds that
	\begin{equation}\label{esjn}
	\begin{aligned}
			&A(t)I(i,k,n)\\\leq& \epsilon\int_{0}^{t}(1+\tau)^{\alpha}\norm{\nabla (|u|^{\frac{p}{2}})}_{L^2(\mathbb{R}^n)}^2(\tau)d\tau+C_{\epsilon}A^{\frac{1+\gamma p}{1+\gamma k}}(t)\int_{0}^{t}(1+\tau)^{\alpha-\frac{1+\gamma p}{1+\gamma k}i}\norm{u}_{L^q(\mathbb{R}^n)}^{\frac{p-k}{1+\gamma k}}d\tau,
	\end{aligned}
	\end{equation}
	where $\gamma=\frac{n}{2}(\frac{1}{q}-\frac{1}{p})$, and $C_\epsilon$ is a positive constant depending on $p,q,n$ and $\epsilon$. 
	
	Specially, if $n=1$ and choose $q=1$, we have
	\begin{equation}\label{esj1}
		\begin{aligned}
		&A(t)I(i,k,1)\\\leq& \epsilon\int_{0}^{t}(1+\tau)^{\alpha}\norm{\nabla (|u|^{\frac{p}{2}})}_{L^2(\mathbb{R})}^2(\tau)d\tau+C_{\epsilon}A^{\frac{1+\gamma p}{1+\gamma k}}(t)\int_{0}^{t}(1+\tau)^{\alpha-\frac{p(p+1)}{2p+k(p-1)}i}\norm{u}_{L^1(\mathbb{R})}^{\frac{2p(p-k)}{2p+k(p-1)}}d\tau.
		\end{aligned}
	\end{equation}
	If $n=2$ and choose $q=6$, we have
	\begin{equation}\label{esj3}
		\begin{aligned}
		&A(t)I(i,k,2)\\\leq& \epsilon\int_{0}^{t}(1+\tau)^{\alpha}\norm{\nabla (|u|^{\frac{p}{2}})}_{L^2(\mathbb{R}^2)}^2(\tau)d\tau+C_{\epsilon}A^{\frac{1+\gamma p}{1+\gamma k}}(t)\int_{0}^{t}(1+\tau)^{\alpha-\frac{p^2}{6p+k(p-6)}i}\norm{u}_{L^6(\mathbb{R}^2)}^{\frac{6p(p-k)}{6p+k(p-6)}}d\tau.\end{aligned}
	\end{equation}

\end{Lem}
\begin{proof}
	By \eqref{esqn}, we have
	\begin{equation}
		\begin{aligned}
		A(t)I(i,k,n)&\leq	CA(t)\int_{0}^{t}(1+\tau)^{\alpha-i}\norm{\nabla (|u|^{\frac{p}{2}})}_{L^2(\mathbb{R}^n)}^{\frac{2\gamma(p-k)}{1+\gamma p}}\norm{u}_{L^q(\mathbb{R}^n)}^{\frac{p-k}{1+\gamma p}}(\tau)d\tau\\
		&=C\int_{0}^{t}(1+\tau)^{\frac{\gamma(p-k)}{1+\gamma p}\alpha}\norm{\nabla (|u|^{\frac{p}{2}})}_{L^2(\mathbb{R}^n)}^{\frac{2\gamma(p-k)}{1+\gamma p}}A(t)(1+\tau)^{\frac{1+rk}{1+\gamma p}\alpha-i}\norm{u}_{L^q(\mathbb{R}^n)}^{\frac{p-k}{1+\gamma p}}(\tau)d\tau\\
		&\leq\epsilon\int_{0}^{t}(1+\tau)^{\alpha}\norm{\nabla (|u|^{\frac{p}{2}})}_{L^2}^2d\tau+C_{\epsilon}A^{\frac{1+\gamma p}{1+\gamma k}}(t)\int_{0}^{t}(1+\tau)^{\alpha-\frac{1+\gamma p}{1+\gamma k}i}\norm{u}_{L^q}^{\frac{p-k}{1+\gamma k}}d\tau,\end{aligned}
	\end{equation}
	where we have used the Young inequality with $\frac{\gamma(p-k)}{1+\gamma p}+\frac{1+\gamma k}{1+\gamma p}=1$, thus the proof of \cref{est} is completed.
\end{proof}

The following lemma gives some properties for $Z(x,y)$.

\begin{Lem}\label{LZ}
	Suppose $Z(x,y)$ is the implicit function determined by \eqref{imz}, then it satisfies that
	
	\begin{itemize}
		
		\item[(i).]  $Z_x(x,y)=\frac{-\varphi'(x-Z(x,y))}{1-\varphi'(x-Z(x,y))},~Z_y(x,y)=\frac{1}{1-\varphi'(x-Z(x,y))}>0$ and \begin{equation}\label{1z}Z_x(x,y)+Z_y(x,y)=1;\end{equation}
		\item[(ii).] $Z_{xx}(x,y)=Z_{yy}(x,y)=-Z_{xy}(x,y)=\frac{-\varphi''(x-Z(x,y))}{(1-\varphi'(x-Z(x,y))^3};$
		
		\item[(iii).] $x-Z(x,y)=G(x-y)$, where $G(\xi)$ is an implicit function determined by
		\begin{equation}\label{G}
			-\xi+G-\varphi (G)=0,
		\end{equation}
		and \begin{equation}\label{Ga}
			G'(\xi)=\frac{1}{1-\varphi'(G(\xi))}>0;
		\end{equation}
		\item[(iv). ] ${\rm sgn}\ (Z(x,y))={\rm sgn}\ (y-\varphi (x))$.
	\end{itemize}
\end{Lem}
\begin{proof}
	Conclusions (i) and (ii) can be obtained by direct calculation.  By rewriting  \eqref{imz} as 
	\begin{equation}\label{pG}-(x-y)+x-Z(x,y)-\varphi(x-Z(x,y))=0,	\end{equation}
	we can get (iii), and the existence of implicit function $G$ can be obtained by \eqref{2DH}. 
	
	To prove (iv), according to \eqref{imz}, we have
	\begin{align}
		y-\varphi(x)=&Z(x,y)-\varphi(x)+\varphi(x-Z(x,y))\\
		=&Z(x,y)-\varphi'(\theta)Z(x,y)\\
		=&Z(x,y)(1-\varphi'(\theta)),
	\end{align}
	since $1-\varphi'(\theta)>0$, thus we get (iv).
\end{proof}

  \subsection{Stability of 2D non-self-similar rarefaction wave to inviscid Burgers equation under perturbation to initial discontinuity}

For the Riemann problem \eqref{eq0}-\eqref{0data} to 2D inviscid Burgers equation, we consider its stability under perturbation to the initial discontinuity, that is, for two different initial discontinuities $y=\varphi_1(x)$ and $y=\varphi_2(x)$.
Let 
 $u^R_i(t,x,y)~(i=1,2)$ be 2D non-self-similar rarefaction wave solution to Riemann problem \eqref{eq0}-\eqref{0data} corresponding to the initial discontinuity $y=\varphi_i(x)~(i=1,2)$, that is
 \begin{equation}\label{Ris}
 		u_i^R(t,x,y)=\begin{cases}
 		u_-,&y< u_-t+\varphi_i (x-u_-t),\\[2mm]
 		\frac{Z_i(x,y)}{t},&u_-t+\varphi _i(x-u_-t)\leq y\leq u_+t+\varphi_i (x-u_+t), \\[1mm]
 		u_+,&y>u_+t+\varphi_i (x-u_+t),
 	\end{cases}
 \end{equation}
 where $Z_i(x,y)~(i=1,2)$ is the implicit function determined by 
 \begin{equation*}
 	y-Z_i-\varphi_i (x-Z_i)=0.
 \end{equation*}
 
 Then we have the following stability result of 2D non-self-similar rarefaction wave under $L^\infty$-perturbation to the initial discontinuity. 

 \begin{Prop}\label{prop12}
 	Suppose both two initial discontinuity curves $\varphi_1(x)\in C^2(\mathbb{R})$ and $\varphi_2(x)\in C^2(\mathbb{R})$ satisfy \cref{2DH}. If 
 	$$\norm{\varphi_1(x)-\varphi_2(x)}_{ L^{\infty}(\mathbb R)}<+\infty,$$
 	then 
 	\begin{equation}\label{RR12}
 		\norm{u^R_1(t,x,y)-u^R_2(t,x,y)}_{L^{\infty}(\mathbb R^2)}\leq \frac{C}{t},~t>0,
 	\end{equation}
 	where $u^R_i(t,x,y)(i=1,2)$ is the rarefaction wave solution to the Riemann problem \eqref{eq0}\eqref{0data} with initial discontinuity $y-\varphi_i(x)=0(i=1,2)$, and $C$ is a positive constant depending only on $d_0$ and $\norm{\varphi_1-\varphi_2}_{L^{\infty}(\mathbb R)}$.
 \end{Prop}
\begin{Rem}
	the result in \cref{prop12} can be extended to the more general scalar conservation laws
	\begin{equation}\label{fg}
		u_t+f(u)_x+g(u)_y=0,
	\end{equation}
	under the assumption $g''(u)-\varphi'(x)f''(u)>d_0>0$ for any $u\in[u_-,u_+],~x\in \mathbb{R}$, and the non-self-similar rarefaction wave solution of \eqref{fg}\eqref{0data} is 
	\begin{equation}
		u^R(t,x,y)=\begin{cases}
			u_-,&y< g'(u_-)t+\varphi (x-f'(u_-)t),\\[2mm]
			c(t,x,y),&g'(u_-)t+\varphi(x-f'(u_-)t)\leq y\leq g'(u_+)t+\varphi (x-f'(u_+)t), \\[1mm]
			u_+,&y>g'(u_+)t+\varphi (x-f'(u_+)t),
		\end{cases}
	\end{equation} 
where $c(t,x,y)$ is an implicit function determined by $y-g'(c)t+\varphi (x-f'(c)t)=0$.
\end{Rem}	

Next, the following lemma demonstrates the effect of an $H^1$ perturbation to initial discontinuity on the initial data $w_0(Z(x,y))$.
  \begin{Lem}\label{Z12}
 	Suppose that both $\varphi_1(x)$ and $\varphi_2(x)$ satisfy \cref{2DH}, and there exists a constant $C_0>0$ such that
 	\begin{align}\label{idh1}
 		\norm{\varphi_1(x)-\varphi_2(x)}_{H^1(\mathbb R)}<C_0,~|\varphi_2''(x)|\leq C_0,
 	\end{align}
 	then $$\norm{w_0(Z_1(x,y))-w_0(Z_2(x,y))}_{H^1(\mathbb R^2)}<+\infty.$$
 	where $w_0(\cdot)$ is defined in \eqref{www1}.
 \end{Lem}
 
For brevity, the proofs of \cref{prop12} and \cref{Z12} will be given in the Appendix. 
  
 \subsection{The reformulation of problem}\label{SEC32}
In this subsection we reformulate our problem by \cref{prop12} and \cref{Z12}. 

Define
	\begin{equation}\label{nnn}L(x):=\begin{cases}
		k_1x+c_1,~x\leq0,\\
		k_2x+c_2,~x>0,
	\end{cases}\end{equation} 

Then we mollify the discontinuous poly-line $y=L(x)$ by defining 
\begin{equation}\label{Lc}
	L_{\epsilon_0}(x)=\alpha_{\epsilon_0}\ast L(x)=\int_{\mathbb R}\alpha_{\epsilon_0}(y)L(x-y)dy,
\end{equation}
where $\alpha_{\epsilon_0}(x)\in C^{\infty}(\mathbb R)$ is the standard mollifier in $\mathbb R$ satisfying
$$
supp~\alpha_{\epsilon_0}(x)=\{x\in \mathbb R:~|x|\leq \epsilon_0\},~\alpha_{\epsilon_0}(x)=\alpha_{\epsilon_0}(-x),~\int_{\mathbb R} \alpha_{\epsilon_0}(x)dx=1.
$$
Then $L_{\epsilon_0}(x)\in C^{\infty}(\mathbb R)$ and
\begin{equation}\label{Le0}
	L_{\epsilon_0}'(x)=
	\left\{\begin{array}{ll}
		\di  k_1,&\di x\leq -\epsilon_0,\\[1mm]
		\di k_2,&\di x\geq \epsilon_0.
	\end{array}
	\right.
\end{equation}
 Combining \eqref{cl} and \eqref{Lc}, it holds that \begin{equation}\label{vh0}\norm{\varphi(x)-L_\epsilon(x)}_{H^1(\mathbb R)}<+\infty.\end{equation} 
Next, let $\varphi_1(x)=\varphi(x),~\varphi_2(x)=L_{\epsilon_0}(x)$, then from \cref{Z12}, it holds that
 \begin{equation}\label{ww0}\norm{w_0(Z_1(x,y))-w_0(Z_2(x,y))}_{H^1(\mathbb R^2)}<+\infty,\end{equation} 
and from \cref{prop12}, we can get
\begin{equation}\label{uut}\norm{u^R_1(t,x,y)-u^R_2(t,x,y)}_{L^{\infty}(\mathbb R^2)}\leq\frac{C}{t},
\end{equation} 
which
 implies that if \cref{mt} holds for the special initial discontinuity curve $y=L_{\epsilon_0}(x)$ defined in \eqref{Lc},
then \cref{mt} is still valid for the initial discontinuity curve $y=\varphi(x)\in C^2(\mathbb{R})$ satisfying \eqref{2DH}  and \eqref{cl}. We will prove the above  assertion in Subsection 3.3.

From now on, we only focus on proving \cref{mt} for the initial discontinuity curve $y=L_{\epsilon_0}(x)$ in \eqref{Lc}. Without ambiguity, we still denote $L_{\epsilon_0}(x)$ as $\varphi(x)$, then 
\begin{equation}\label{var1}
	\varphi'(x)=\begin{cases}k_1,&\di ~x\leq-\epsilon_0,\\
		L'_{\epsilon_0}(x),&\di ~-\epsilon_0<x<\epsilon_0,\\
		k_2,&\di ~x\geq\epsilon_0,
	\end{cases}
\end{equation}
and
\begin{equation}\label{var2}
	\varphi''(x)=\begin{cases}
		L''_{\epsilon_0}(x),&\di ~\left| x\right| <\epsilon_0,\\
		0,&\di ~\left| x\right| \geq\epsilon_0.
	\end{cases}
\end{equation}
	
 By observing the form of 2D non-self-similar rarefaction wave \eqref{Rs1} and  \cref{LZ} (iii), we introduce the following variable substitutions 
\begin{equation}\label{V}
	\xi=x-y,~\eta=Z(x,y),
\end{equation}
and the Jacobi determinant of  the transformation \eqref{V} is 
\begin{align}
	\di \frac{\partial (\xi, \eta)}{\partial(x,y)}&\di =\begin{vmatrix}
		\partial_x(x-y)&\partial_y(x-y)\\\partial_xZ(x,y))&\partial_yZ(x,y)
	\end{vmatrix}=\begin{vmatrix}
		1&-1\\Z_x(x,y)&Z_y(x,y)
	\end{vmatrix}
	=1.
\end{align}
Therefore, under the new spatial coordinate $(\xi,\eta)$,  2D non-self-similar rarefaction wave \eqref{Rs1} becomes  1D planar self-similar rarefaction wave
\begin{equation}\label{Rs3}
	{u}^R(\frac{\eta}{t})=\begin{cases}
		u_-,&\di ~\eta<u_-t,\\ 
	\di \frac{\eta}{t},&\di ~u_-t\leq \eta\leq u_+t,\\
		u_+,&\di ~\eta>u_+t.
	\end{cases}
\end{equation}
Then by the time-asymptotic stability and decay rate for planar self-similar rarefaction wave in \cite{Xin1990, Ito1996}), we can construct a smooth ansatz $w(t,\eta)$
\begin{equation}\label{w1}
\left\{\begin{array}{l}
\di 	w_t+ww_{\eta}=0,\\
\di w(0,\eta)=w_0(\eta)=\frac{u_++u_-}{2}+\frac{u_+-u_-}{2}\kappa \int_{0}^{\eta}\frac{1}{1+\alpha^2}d\alpha ,
\end{array}
	\right.
\end{equation}
such that 
\begin{equation}\label{T1}\norm{w(t,\eta)-u^R(\frac{\eta}{t})}_{L^\infty(\mathbb R^2)}\leq C(1+t)^{-\frac{1}{2}}.\end{equation}
 
 On the other hand, by using the properties of $Z(x,y)$ in \cref{LZ}, 2D viscous Burgers equation \eqref{eq} is transformed into
\begin{equation}\label{equ}
	u_t+uu_{\eta} =K(\xi)u_{\eta\eta}+A(\xi)u_{\xi\eta}+2u_{\xi\xi}+B(\xi)u_{\eta},
\end{equation}
where $K(\xi),A(\xi)$ and $B(\xi)$ are given functions that only depend on variable $\xi$, defined in \eqref{coeff}.  From \eqref{2DH}, there exist positive constants
$K_1,K_2$ and $M$ such that
\begin{equation}\label{KAB}
	0<K_1\leq K(\xi)\leq K_2,~ \left| A(\xi)\right| \leq M, ~\left| B(\xi)\right| \leq M,
\end{equation}
where $K_1,K_2, M$ only depend on $d_0$.

 \section{The proof of \cref{mt}}
 In this section, we prove our main result \cref{mt} for the time-asymptotic stability and the time decay rate of 2D non-self-similar rarefaction wave \eqref{Rs} or \eqref{Rs1} to 2D viscous Burgers equation \eqref{eq}.
Due to \eqref{T1},  it is sufficient to prove the time-asymptotic stability and the time decay rate of approximate rarefaction wave $w(t,\eta)$ to the transformed 2D Burgers equation \eqref{equ} in the new spatial coordinate $(\xi,\eta)$. Then the main difficulties lie in the error terms for the inviscid approximate rarefaction wave $w(t,\eta)$ to the transformed 2D Burgers equation \eqref{equ} with variable and mixed-derivative viscosities.

 
\subsection{Construction of 2D viscous profile and its properties}
In this subsection, we will construct the 2D viscous rarefaction profile and recover the viscosity term $K(\xi) u_{\eta\eta}$.
 Motivated by Xin \cite{Xin1990}, we construct the following approximate rarefaction wave $v(t,\xi,\eta)$ :
\begin{equation}\label{v1}
 \left\{\begin{array}{ll}
	\di  	v_t+vv_{\eta}=K(\xi)v_{\eta\eta},\\[2mm]
	\di	v(0,\xi,\eta):=w_0(\eta),
\end{array}
\right.	
\end{equation}
where the initial data $w_0(\eta)$ is exactly same as the one for inviscid rarefaction wave profile in \eqref{w1}.
Note that since the viscosity coefficient $K(\xi)$ in \eqref{v1} depends on the variable $\xi$,  the corresponding solution also depend on $\xi$ as $v(t,\xi,\eta)$.
 
 If we fix $\xi\in \mathbb{R}$, then \eqref{v1} can be seen as 1D viscous Burgers equation.
 According to the 1D asymptotic stability result in \cite{HN1991}, we have
\begin{equation}\label{T2}
	\norm{v(t,\xi,\eta)-w(t,\eta)}_{L^{\infty}(\mathbb R^2)}\leq C (1+t)^{-\frac{1}{2}}.
\end{equation}
Moreover, for each fixed $\xi\in \mathbb{R}$, the following \cref{vl1} lists some other properties for $v(t,\xi,\eta)$, whose proof can be found in \cite{HN1991,KT2004}.

\begin{Lem}\label{vl1} The smooth solution $v(t,\xi,\eta)$ satisfies the following properties:

\begin{itemize}
	\item[(i)]  $u_-< v(t,\xi,\eta)< u_+$, and $v_{\eta}(t,\xi,\eta)>0$.
	
	\item[(ii)] There exists a uniform-in-time positive constant $C$ such that
	\begin{equation}\label{vbp}
		\norm{v_{\eta}}_{L_{\eta}^p(\mathbb R)}\leq C(1+t)^{-1+\frac{1}{p}},~ \forall p\in [1,+\infty],
	\end{equation}
	\begin{equation}\label{vbbp}
		\norm{v_{\eta\eta}}_{L_{\eta}^p(\mathbb R)}\leq C(1+t)^{-\frac{3}{2}+\frac{1}{2p}},~\forall p\in [1,+\infty].
	\end{equation}
\end{itemize}	
\end{Lem}
 By \eqref{KAB}, the viscosity coefficient $K(\xi)$ have the uniform positive upper and lower bounds $K_1,K_2$ respectively. Then we can choose the constant $C$ in \cref{vl1} (ii) depending on $K_1,K_2$, but independent of $\xi\in \mathbb{R}$.

Nevertheless, we shall prove some properties of $v(t,\xi,\eta)$ with respect to the variable $\xi$.
First, $v_{\xi}(t,\xi,\eta)$ satisfies the equation
\begin{equation}\label{vae}
	\left\{\begin{array}{ll}
		\di v_{\xi t}+v_{\xi}v_{\eta}+vv_{\xi\eta}=K(\xi)v_{\xi \eta\eta}+K'(\xi)v_{\eta\eta},\\[2mm]
		\di	v_{\xi}(0,\xi,\eta)=0.
	\end{array}
	\right.
\end{equation}
Then we have the following lemma for the properties of $v_{\xi}(t,\xi,\eta)$.
\begin{Lem}\label{vl2} For the smooth solution $v_{\xi}(t,\xi,\eta)$ of \eqref{vae}, it holds that
	
	\begin{itemize}
		\item[{\rm (i)}.]  Denote $G_1:=G^{-1}(-\epsilon_0)$ and $G_2:=G^{-1}(\epsilon_0)$. If the initial discontinuity curve $y=\varphi(x)$ satisfies \eqref{cl}, then the support of $v_{\xi}(t,\xi,\eta)$ satisfies
		\begin{equation}\label{sva}
			 supp~v_{\xi}(t,\xi,\eta)\subseteq \mathbb R_+\times D,~ D\triangleq\{(\xi,\eta)\in\mathbb R^2|~\xi\in[G_1,G_2],~\eta\in\mathbb R\},
			\end{equation}
			where the $\epsilon_0$ is defined in the mollifier $\alpha_{\epsilon_0}(x)$ in \eqref{Lc},
		\item[{\rm (ii)}.] There exists a generic positive constant $C=C(p,u_-,u_+,G_1,G_2)$ such that
	\begin{align}\label{va1}
		& \di \norm{v_{\xi}}_{L^1(\mathbb R^2)}\leq C\ln(1+t),\\ \label{vaa1}
		& \di \norm{v_{\xi\xi}}_{L^1(\mathbb R^2)}\leq C\ln(1+t),\\
	\label{vap}
			& \di \norm{v_{\xi}}_{L^p(\mathbb R^2)}\leq C(1+t)^{-\frac{1}{2}+\frac{1}{2p}}\ln(1+t),~2\leq p<+\infty,	\\
			\label{vainf}
			& \di	\norm{v_{\xi}}_{L^{\infty}(\mathbb R^2)}\leq C(1+t)^{-\frac{1}{2}}\ln(3+t),
			\\
			\label{vabp}
			& \di \norm{v_{\xi\eta}}_{L^p(\mathbb R^2)}\leq C(1+t)^{-1+\frac{1}{2p}}\ln(3+t),~2\leq p<+\infty,\\
	\label{vaap}
				& \di \norm{v_{\xi\xi}}_{L^p(\mathbb R^2)}\leq C(1+t)^{-\frac{1}{2}+\frac{1}{2p}}\ln^{3}(3+t),~2\leq p<+\infty.
		\end{align}
	\end{itemize}	
\end{Lem}

\begin{proof}
First, we prove property (i). From \eqref{coeff} and \eqref{Ga}, we have
\begin{equation}\label{Kd}
K^\prime(\xi)=\frac{2\big(1+\varphi'(G(\xi))\big)\varphi''(G(\xi))G^\prime(\xi)}{(1-\varphi'(G(\xi)))^3}=\frac{2(1+\varphi'(G(\xi)))}{(1-\varphi'(G(\xi)))^4}\varphi''(G(\xi)).
\end{equation}
By \eqref{var2}, if $\left| G(\xi)\right| \geq \epsilon_0$, then $\varphi''(G(\xi))=0$,  thus \begin{equation}\label{G0}K'(\xi)\equiv0~\text{for}~ \left| G(\xi)\right| \geq \epsilon_0.\end{equation}
Using the Hopf-Cole transformation,  the solution to \eqref{v1} can be expressed by
$$v(t,\xi,\eta)=\frac{\di \int_{\mathbb{R}}\frac{\eta-y}{t}\mathrm {e}^{-\frac{(\eta-y)^2}{4K(\xi)t}-\frac{1}{2K(\xi)}\di \int_{0}^{y}w_0(z)dz}dy}{\di \int_{\mathbb{R}}\mathrm {e}^{-\frac{(\eta-y)^2}{4K(\xi)t}-\frac{1}{2K(\xi)}\di \int_{0}^{y}w_0(z)dz}dy},$$
then
$$
\begin{array}{ll}
\di \frac{\partial v}{\partial K}=&\frac{\di \int_{\mathbb{R}}\frac{\eta-y}{t}\mathrm {e}^{-\frac{(\eta-y)^2}{4K(\xi)t}-\frac{1}{2K(\xi)}\di \int_{0}^{y}w_0(z)dz}\bigg(\frac{(\eta-y)^2}{4K^2(\xi)t}+\frac{\di \int_{0}^{y}w_0(z)dz}{2K^2(\xi)}\bigg)dy}{\di \int_{\mathbb{R}}\mathrm {e}^{-\frac{(\eta-y)^2}{4K(\xi)t}-\frac{1}{2K(\xi)}\di \int_{0}^{y}w_0(z)dz}dy}
\\
&-\quad\frac{\di\int_{\mathbb{R}}\mathrm {e}^{-\frac{(\eta-y)^2}{4K(\xi)t}-\frac{1}{2K(\xi)}\di \int_{0}^{y}w_0(z)dz}\bigg(\frac{(\eta-y)^2}{4K^2(\xi)t}+\frac{\di \int_{0}^{y}w_0(z)dz}{2K^2(\xi)}\bigg)dy}{\di \bigg(\int_{\mathbb{R}}\mathrm {e}^{-\frac{(\eta-y)^2}{4K(\xi)t}-\frac{1}{2K(\xi)}\di \int_{0}^{y}w_0(z)dz}dy\bigg)^2}.
\end{array}
$$
Since $0<K_1\leq K(\xi)\leq K_2$ for positive constants $K_1$ and $K_2$, $\frac{\partial v}{\partial K}$ is smooth and bounded for $\xi \in \mathbb{R}.$ By the chain rule
\begin{equation}\label{vkk}
	v_{\xi}=\frac{\partial v}{\partial K}K'(\xi),
\end{equation}
 we can prove the property (i) by combining \eqref{G0} and \eqref{vkk}.

Next, the proof of property (ii) is divided into the following four steps.

\noindent$\mathbf{\underline{Step~1}.~}$ We first prove the $L^1$-estimate \eqref{va1}. 

Motivated by \cite{Ito1996}, we let $j_\delta(x)$ be the standard mollifier in $\mathbb{R}$ and $sgn(\tau)$ be the sign function.  Set \begin{equation*} s_\delta(\lambda) :=j_{\delta}\convolution sgn(\lambda)=\int_{\mathbb R}j_{\delta}(\lambda-\tau)sgn(\tau)d\tau,\end{equation*}
and
\begin{equation*}
\Phi_\delta(\lambda) :=\int_{0}^{\lambda}s_{\delta}(\beta)d\beta.
\end{equation*}
Then
\begin{equation}\label{phid}
	s_\delta'(\lambda)=2j_{\delta}(\lambda).
\end{equation}
Multiplying \eqref{vae} by $s_\delta(v_{\xi})$ and then integrating the resulting equation over $\mathbb R^2$,
we have 
\begin{align*}
	\frac{d}{dt}\iint_{\mathbb R^2}\Phi_\delta(v_{\xi})d\xi d\eta +\iint_{\mathbb R^2}s_\delta(v_{\xi})(vv_{\xi})_{\eta}d\eta d\xi =&\iint_{\mathbb R^2}K(\xi)[(s_\delta(v_{\xi})v_{\xi\eta})_{\eta}-2j_{\delta}(v_{\xi})v_{\xi\eta}^2]d\eta d\xi \\
	&+\iint_{\mathbb R^2}K'(\xi)s_\delta(v_{\xi})v_{\eta\eta}d\eta d\xi.
\end{align*}
By \eqref{sva}, we have
\begin{equation}\label{L1e}
\begin{array}{ll}
	\di \quad\frac{d}{dt}\iint_{\mathbb R^2}\Phi_\delta(v_{\xi})d\xi d\eta +\iint_{\mathbb R^2}2K(\xi)j_{\delta}(v_{\xi})v_{\xi\eta}^2d\eta\\[6mm]
	 \di =
\iint_{D}2j_\delta(v_{\xi})vv_{\xi}v_{\xi\eta}d\eta d\xi +\iint_{D}K'(\xi)s_\delta(v_{\xi})v_{\eta\eta}d\eta d\xi \\[6mm]
\di =\iint_{D}\left(\int_{0}^{v_{\xi}}2j_\delta(\alpha)\alpha d\alpha\right)_{\eta}vd\eta d\xi +\iint_{D}K'(\xi)s_\delta(v_{\xi})v_{\eta\eta}d\eta d\xi 
\\[6mm]
\di =-\iint_{D}\left(\int_{0}^{v_{\xi}}2j_\delta(\alpha)\alpha d\alpha\right)v_{\eta}d\eta d\xi +\iint_{D}K'(\xi)s_\delta(v_{\xi})v_{\eta\eta}d\eta d\xi 
\\[6mm]
\di \leq C\delta \norm{v_{\eta}}_{L^1(D)}+C\norm{v_{\eta\eta}}_{L^1(D)}\\[2mm]
\di \leq C\delta+C(1+t)^{-1},
\end{array}
\end{equation}
where we have used the fact
\begin{equation}\label{delta}
\left|\int_{0}^{v_{\xi}}j_\delta(\alpha)\alpha d\alpha\right|\leq\delta \int_{0}^{+\infty}j(\alpha)\alpha d\alpha\leq C\delta.
\end{equation}
Integrating \eqref{L1e}  from $0$ to $t$ yields that
\begin{align*}\iint_{\mathbb R^2}\Phi_\delta(v_{\xi})d\xi d\eta \leq C \delta t+C\ln (1+t).
\end{align*}
Then passing the limit $\delta \rightarrow 0+$, we prove the estimate \eqref{va1}.

\noindent$\mathbf{\underline{Step~2}.~}$ We prove $L^p$-estimate \eqref{vap} for $v_{\xi}$. 


Multiplying \eqref{vae} by $|v_{\xi}|^{p-2}v_{\xi}$ and integrating the resulting equation over $\mathbb R$ with respect to $\eta$,
we have 
\begin{align*}
	\frac{1}{p}\frac{d}{dt}\norm{v_{\xi}}_{L_{\eta}^p(\mathbb R))}^p+\frac{p-1}{p}\int_{\mathbb R}v_{\eta}|v_{\xi}|^pd\eta+(p-1)\int_{\mathbb R}K(\xi)|v_{\xi}|^{p-2}v_{\xi\eta}^2d\eta=\int_{\mathbb R}K'(\xi)v_{\eta\eta}|v_{\xi}|^{p-2}v_{\xi}d\eta.
	 \end{align*}
By \eqref{vbbp}, one has
\begin{align*}
	\int_{\mathbb R}K'(\xi)v_{\eta\eta}|v_{\xi}|^{p-2}v_{\xi}d\eta&\leq 	C\norm{v_{\eta\eta}}_{L_{\eta}^p(\mathbb R)}\norm{v_{\xi}}_{L_{\eta}^p(\mathbb R)}^{p-1}\\
	&\leq C(1+t)^{-\frac{3}{2}+\frac{1}{2p}}\norm{v_{\xi}}_{L_{\eta}^p(\mathbb R)}^{p-1}\\&\leq C(1+t)^{-1}\norm{v_{\xi}}_{L_{\eta}^p(\mathbb R)}^{p}+C(1+t)^{-\frac{p}{2}-\frac{1}{2}}.
	\end{align*}
Then we have
\begin{equation}\label{eqno1t}
\begin{aligned}
	&\frac{d}{dt}\norm{v_{\xi}}_{L_{\eta}^p(\mathbb R))}^p+(p-1)\int_{\mathbb R}v_{\eta}|v_{\xi}|^pd\eta+C\int_{\mathbb R}|v_{\xi}|^{p-2}v_{\xi\eta}^2d\eta\\
	\leq &C(1+t)^{-1}\norm{v_{\xi}}_{L_{\eta}^p(\mathbb R)}^{p}+C(1+t)^{-\frac{p}{2}-\frac{1}{2}}.
\end{aligned}
\end{equation}
By multiplying \eqref{eqno1t} by $(1+t)^\alpha$, where $\alpha$ is some sufficiently large positive constant, and then integrating the resulting equation from $0$ to $t$, and noting that $v_{\xi}|_{t=0}=0$, we have
\begin{align*}
	&(1+t)^\alpha\norm{v_{\xi}}_{L_{\eta}^p(\mathbb R))}^p+C_p\int_{0}^{t}\int_{\mathbb R}(1+\tau)^\alpha v_{\eta}|v_{\xi}|^pd\eta d\tau+C\int_{0}^{t}\int_{\mathbb R}(1+\tau)^\alpha |v_{\xi}|^{p-2}v_{\xi\eta}^2d\eta d\tau\\
	&\leq 2\int_{0}^{t}(1+\tau)^{\alpha-1}\norm{v_{\xi}}_{L_{\eta}^p(\mathbb R))}^pd\tau+C(1+t)^{\alpha-\frac{p}{2}+\frac{1}{2}}\\
	&\triangleq J_1+C(1+t)^{\alpha-\frac{p}{2}+\frac{1}{2}},
\end{align*}
By 
 \eqref{esj1}, we have
\begin{align*}
	J_1=CI(1,0,1)&\leq \epsilon\int_{0}^{t}(1+\tau)^{\alpha}\norm{\partial_{\eta} (|v_{\xi}|^{\frac{p}{2}})}_{L^2(\mathbb{R})}^2(\tau)d\tau+C_{\epsilon}\int_{0}^{t}(1+\tau)^{\alpha-\frac{p+1}{2}}\norm{v_{\xi}}_{L^1(\mathbb{R})}^p(\tau)d\tau\\
	&\leq \epsilon\int_{0}^{t}(1+\tau)^{\alpha}\norm{\partial_{\eta} (|v_{\xi}|^{\frac{p}{2}})}_{L^2(\mathbb{R})}^2(\tau)d\tau+C_{\epsilon}(1+t)^{\alpha-\frac{p}{2}+\frac{1}{2}}\ln^p(1+t).
\end{align*}
Therefore, we can get
\begin{align*}(1+t)^\alpha\norm{v_{\xi}}_{L_{\eta}^p(\mathbb R))}^p\leq C(1+t)^{\alpha-\frac{p}{2}+\frac{1}{2}}\ln^p(1+t)+C(1+t)^{\alpha-\frac{p}{2}+\frac{1}{2}},
\end{align*}
and then
\begin{align}\label{lp1}
	\norm{v_{\xi}}_{L_{\eta}^p(\mathbb R))}^p\leq C(1+t)^{-\frac{p}{2}+\frac{1}{2}}\ln^p(1+t),~p\geq 2.
\end{align}
Due to \eqref{sva}, we have
\begin{align}\label{lp2}
	\norm{v_{\xi}}_{L^p(\mathbb R^2)}^p=\iint_{D}|v_{\xi}|^pd\xi d\eta =\int_{G_1}^{G_2}\norm{v_{\xi}}_{L_{\eta}^p(\mathbb R))}^pd\xi
	\leq C_{G_1,G_2}(1+t)^{-\frac{p}{2}+\frac{1}{2}}\ln^p(1+t),
\end{align}
which complete the proof of \eqref{vap}.   

\noindent$\mathbf{\underline{Step~3}.~}$ We prove $L^p$-estimate \eqref{vabp} for $v_{\xi\eta}$ and $L^{\infty}$-estimate \eqref{vainf} for $v_{\xi}$. 

Differentiating \eqref{vae} with respect to $\eta$, and putting $W:=v_{\xi\eta}$, we can get
\begin{equation}\label{vab}
	W_t+2v_{\eta}W+vW_{\eta}+v_{\xi}v_{\eta\eta}=K(\xi)W_{\eta\eta}+K'(\xi)v_{\eta\eta\eta},
\end{equation}
Multiplying \eqref{vab} by $|W|^{p-2}W$ and
integrating the resulting equation over $\mathbb R$ with respect to $\eta$, we have
\begin{equation}\label{j12}
\begin{array}{ll}
	\di \frac{1}{p}\frac{d}{dt}\norm{W}_{L_{\eta}^p(\mathbb R))}^p+\frac{p-1}{p}\int_{\mathbb R}v_{\eta}|W|^pd\eta+(p-1)\int_{\mathbb R}K(\xi)|W|^{p-2}W_{\eta}^2d\eta\\[4mm]
	\di =(p-1)\int_{\mathbb R}v_{\eta}v_{\xi}|W|^{p-2}W_{\eta}d\eta-(p-1)\int_{\mathbb R}K'(\xi)v_{\eta\eta}|W|^{p-2}W_{\eta}d\eta\\[4mm]
	\di :=J_2+J_3.
\end{array}
\end{equation}
First, we have
\begin{align*}
	J_2&\leq\norm{v_{\eta}}_{L^{\infty}}\norm{v_{\xi}}_{L^p}\norm{W}_{L^p}^{\frac{p-2}{2}}
	\norm{\partial_{\eta} (|W|^{\frac{p}{2}})}_{L^2}\\
	&\leq\norm{v_{\eta}}_{L^{\infty}}\norm{v_{\xi}}_{L^p}^{1+\frac{p-2}{p+2}}
	\norm{\partial_{\eta} (|W|^{\frac{p}{2}})}_{L^2}^{1+\frac{p-2}{p+2}}\\
&\leq\epsilon \norm{\partial_{\eta} (|W|^{\frac{p}{2}})}_{L^2}^2+C_{\epsilon}\norm{v_{\eta}}_{L^{\infty}}^{\frac{p+2}{2}}\norm{v_{\xi}}_{L^p}^p\\
&\leq\epsilon \norm{\partial_{\eta} (|W|^{\frac{p}{2}})}_{L^2}^2+C_{\epsilon}(1+t)^{-p-\frac{1}{2}}\ln^p(1+t),
\end{align*}
where we have used the interpolation inequality \eqref{puin1} for $W$ and H\"{\rm o}lder inequality.
 Then by Young inequality and  the interpolation inequality \eqref{puin1} for $W$ again, we obtain
 \begin{align*}
	J_3&\leq\epsilon \int_{\mathbb R}|W|^{p-2}W_{\eta}^2d\eta+C_{\epsilon}\int_{\mathbb R}|W|^{p-2}v_{\eta\eta}^2d\eta\\
	&\leq\epsilon \int_{\mathbb R}|W|^{p-2}W_{\eta}^2d\eta+C_{\epsilon}\norm{v_{\eta\eta}}_{L^{p}}^2\norm{W}_{L^p}^{p-2}\\
	&\leq\epsilon \int_{\mathbb R}|W|^{p-2}W_{\eta}^2d\eta+C_{\epsilon}\norm{v_{\eta\eta}}_{L^{p}}^2\norm{v_{\xi}}_{L^p}^{\frac{2(p-2)}{p+2}}
	\norm{\partial_{\eta} (|W|^{\frac{p}{2}})}_{L^2}^{\frac{2(p-2)}{p+2}}\\
	&\leq\epsilon \norm{\partial_{\eta} (|W|^{\frac{p}{2}})}_{L^2}^2+C_{\epsilon}\norm{v_{\eta\eta}}_{L^{p}}^{\frac{p+2}{2}}\norm{v_{\xi}}_{L^p}^{\frac{p-2}{2}}\\
		&\leq\epsilon \norm{\partial_{\eta} (|W|^{\frac{p}{2}})}_{L^2}^2+C_{\epsilon}(1+t)^{-p-\frac{1}{2}}\ln^{\frac{p-2}{2}}(1+t).
\end{align*}
Substituting the above two estimates into \eqref{j12} yields that
\begin{align}\label{vabt}
	\frac{d}{dt}\norm{W}_{L_{\eta}^p(\mathbb R))}^p+C_1\int_{\mathbb R}v_{\eta}|W|^pd\eta+C_2\norm{\partial_{\eta} (|W|^{\frac{p}{2}})}_{L^2}^2\leq C(1+t)^{-p-\frac{1}{2}}\ln^{p}(3+t).
\end{align}
Multiplying \eqref{vabt} by $(1+t)^{\alpha}$ with some sufficiently large positive constant $\alpha$ and
integrating the resulting equation from $0$ to $t$, and noting that $v_{\xi\eta}|_{t=0}=0$,
we have
\begin{equation}\label{eee}
\begin{array}{ll}
	\di (1+t)^\alpha\norm{W}_{L_{\eta}^p(\mathbb R))}^p+C\int_{0}^{t}\int_{\mathbb R}(1+\tau)^\alpha v_{\eta}|W|^pd\eta d\tau+C\int_{0}^{t}(1+\tau)^\alpha\norm{\partial_{\eta} (|W|^{\frac{p}{2}})}_{L^2}^2(\tau)d\tau\\[4mm]
	\di \leq C\int_{0}^{t}(1+\tau)^{\alpha-1}\norm{W}_{L_{\eta}^p(\mathbb R)}^p(\tau)d\tau+C(1+t)^{\alpha-p+\frac{1}{2}}\ln^{p}(3+t)\\[4mm]
	\di \triangleq J_4+C(1+t)^{\alpha-p+\frac{1}{2}}\ln^{p}(3+t).
\end{array}
\end{equation}
By the interpolation inequality \eqref{puin1} and \eqref{lp1}, we can get
\begin{align*}
	J_4&\leq C\int_{0}^{t}(1+\tau)^{\alpha-1}\norm{\partial_{\eta} (|W|^{\frac{p}{2}})}_{L^2(\mathbb{R})}^{\frac{2p}{p+2 }}\norm{v_{\xi}}_{L^p(\mathbb{R})}^{\frac{2p}{p+2}}d\tau\\
	&\leq \epsilon \int_{0}^{t}(1+\tau)^\alpha\norm{\partial_{\eta} (|W|^{\frac{p}{2}})}_{L^2}^2+C_{\epsilon}\int_{0}^{t}(1+\tau)^{\alpha-\frac{p+2}{2}}\norm{v_{\xi}}_{L^p(\mathbb{R})}^pd\tau\\
		&\leq \epsilon \int_{0}^{t}(1+\tau)^\alpha\norm{\partial_{\eta} (|W|^{\frac{p}{2}})}_{L^2}^2+C_{\epsilon}(1+t)^{\alpha-p+\frac{1}{2}}\ln^{p}(1+t).
\end{align*}
Substituting the above estimate into \eqref{eee} implies that
\begin{align*}
(1+t)^\alpha\norm{W}_{L_{\eta}^p(\mathbb R))}^p\leq C(1+t)^{\alpha-p+\frac{1}{2}}\ln^{p}(3+t),
\end{align*}
the  \begin{equation}\label{1v}
	\norm{W}_{L_{\eta}^p(\mathbb R))}^p\leq C(1+t)^{-p+\frac{1}{2}}\ln^{p}(3+t).
\end{equation} Similar to \eqref{lp2}, we can obtain
\begin{align*}
\norm{W}_{L^p(\mathbb{R}^2)}^p\leq C(1+t)^{-p+\frac{1}{2}}\ln^{p}(3+t),
\end{align*}
which prove \eqref{vabp}.

Now, we can prove the $L^{\infty}$-estimate \eqref{vainf} for $v_{\xi}$. Fixed the variable $\xi$, then 
by the estimate \eqref{lp1} and \eqref{1v},
we have  \begin{align*}\norm{v_{\xi}}_{L_{\eta}^{\infty}(\mathbb R)}\leq C\norm{v_{\xi}}_{L_{\eta}^2(\mathbb R)}^{\frac{1}{2}}\norm{ v_{\xi\eta}}_{L_{\eta}^2(\mathbb R)}^{\frac{1}{2}}\leq C_{\xi}(1+t)^{-\frac{1}{2}}\ln(3+t).
\end{align*}
where $C_{\xi}$ is a positive constant depending on $\xi$. Due to \eqref{sva},	$ v_{\xi}(t,\xi,\eta)\equiv 0 ~\text{for}~\xi\notin[G_1,G_2].$ So we can choose a constant C depending on $G_1$ and $G_2$, but independent of $\xi$, such that
\begin{align*}\norm{v_{\xi}}_{L^{\infty}(\mathbb R^2)}\leq C(1+t)^{-\frac{1}{2}}\ln(3+t),
\end{align*}
this yields the estimate \eqref{vainf}.

\noindent$\mathbf{\underline{Step~4}.~}$ We prove $L^1$-estimate \eqref{vaa1} and $L^p$-estimate \eqref{vaap} for $v_{\xi\xi}$.

Differentiating \eqref{vae} with respect to $\xi$, and putting $Q:=v_{\xi\xi}$, we obtain
\begin{equation}\label{vaae}
	Q_{t}+v_{\eta}Q+vQ_{\eta}+2v_{\xi}v_{\xi\eta}=K(\xi)Q_{\eta\eta}+2K'(\xi)v_{\xi \eta\eta}+K''(\xi)v_{\eta\eta},
\end{equation}
Similar to Step 1, we using the functions $s_\delta(\lambda)$ and $\Phi_\delta(\lambda)$ to get the estimate $\norm{v_{\xi\xi}}_{L^1_{\eta}(\mathbb R)}$.
Multiplying \eqref{vaae} by $s_\delta(Q)$ and integrating the resulting equation over $\mathbb R$ with respect to $\eta$, we can arrive at
\begin{align*}
	&\frac{d}{dt}\int_{\mathbb R}\Phi_\delta(Q)d\eta+\int_{\mathbb R}s_\delta(Q)(vQ)_{\eta}d\eta+\int_{\mathbb R}s_\delta(Q)(v_{\xi}^2)_{\eta}d\eta\\=&\int_{\mathbb R}K(\xi)[(s_\delta(Q)Q_{\eta})_{\eta}-2j_{\delta}(Q)Q_{\eta}^2]d\eta+\int_{\mathbb R}2K'(\xi)s_\delta(Q)v_{\xi \eta\eta}d\eta+\int_{\mathbb R}K''(\xi)s_\delta(Q)v_{\eta\eta}d\eta.
\end{align*}
Then we have
\begin{align*}
	&\frac{d}{dt}\int_{\mathbb R}\Phi_\delta(Q)d\eta+2\int_{\mathbb R}K(\xi)j_{\delta}(Q)Q_{\eta}^2d\eta\\=&\int_{\mathbb R}2j_{\delta}(Q)Q_{\eta}vQd\eta+\int_{\mathbb R}2j_{\delta}(Q)Q_{\eta}v_{\xi}^2d\eta
	-\int_{\mathbb R}4K'(\xi)j_{\delta}(Q)Q_{\eta}v_{\xi\eta}d\eta+\int_{\mathbb R}K''(\xi)s_\delta(Q)v_{\eta\eta}d\eta\\
	\leq &2\int_{\mathbb R}(\int_{0}^{Q}j_\delta(\alpha)\alpha d\alpha)_{\eta}vd\eta+\epsilon\int_{\mathbb R}j_{\delta}(Q)Q_{\eta}^2d\eta+C_{\epsilon}\int_{\mathbb R}(v_{\xi}^4+v_{\xi\eta}^2)d\eta+C\norm{v_{\eta\eta}}_{L^1}\\
	\leq & C\delta \norm{v_{\eta}}_{L^1}+\epsilon\int_{\mathbb R}j_{\delta}(Q)Q_{\eta}^2d\eta+C_{\epsilon}(1+t)^{-\frac{3}{2}}\ln^4(1+t)+C(1+t)^{-1}
	\\\leq& C\delta+\epsilon\int_{\mathbb R}j_{\delta}(Q)Q_{\eta}^2d\eta+C_{\epsilon}(1+t)^{-1}.
\end{align*}
Integrating the above inequality from $0$ to $t$, we can get
\begin{align*}\int_{\mathbb R}\Phi_\delta(Q)d\eta\leq C \delta t+C\ln(1+t).
\end{align*}
Letting $\delta \rightarrow 0+$, we have
\begin{align}\label{vaa1e}
	\norm{v_{\xi\xi}}_{L^1_{\eta}(\mathbb R)}\leq C\ln(1+t).
\end{align}
Similar to \eqref{lp2}, we can derive  $L^1$-estimate \eqref{vaa1}.

Now we can prove $L^p$-estimate \eqref{vaap} for $v_{\xi\xi}$.
Multiplying \eqref{vab} by $|Q|^{p-2}Q~(p\geq 2)$ and
integrating  the resulting equation over $\mathbb R$ with respect to $\eta$, we have

\begin{align*}
	&\quad\frac{1}{p}\frac{d}{dt}\norm{Q}_{L_{\eta}^p(\mathbb R))}^p+\frac{p-1}{p}\int_{\mathbb R}v_{\eta}|Q|^pd\eta+(p-1)\int_{\mathbb R}K(\xi)|Q|^{p-2}Q_{\eta}^2d\eta\\&=-2\int_{\mathbb R}v_{\xi}v_{\xi\eta}|Q|^{p-2}Qd\eta-2(p-1)\int_{\mathbb R}K'(\xi)v_{\xi\eta}|Q|^{p-2}Q_{\eta}d\eta
	+\int_{\mathbb R}K''(\xi)v_{\eta\eta}|Q|^{p-2}Qd\eta\\
	&\leq \norm{v_{\xi}}_{L^{2p}}\norm{v_{\xi\eta}}_{L^{2p}}\norm{Q}_{L^p}^{p-1}+\epsilon \int_{\mathbb R}|Q|^{p-2}Q_{\eta}^2d\eta +C_{\epsilon}\int_{\mathbb R}\int_{\mathbb R}v_{\xi\eta}^2|Q|^{p-2}d\eta +C\norm{v_{\eta\eta}}_{L^p}\norm{Q}_{L^p}^{p-1}
	\\&\leq \epsilon \int_{\mathbb R}|Q|^{p-2}Q_{\eta}^2d\eta +C_{\epsilon}\norm{v_{\xi\eta}}_{L^p}^2\norm{Q}_{L^p}^{p-2}+C[\ln^2(3+t)+1](1+t)^{-\frac{3}{2}+\frac{1}{2p}}\norm{Q}_{L^p}^{p-1}\\
	&\leq \epsilon \int_{\mathbb R}|Q|^{p-2}Q_{\eta}^2d\eta +C\ln^{2}(3+t)(1+t)^{-2+\frac{1}{p}}\norm{Q}_{L^p}^{p-2}+C(1+t)^{-\frac{3}{2}+\frac{1}{2p}}\ln^{2}(3+t)\norm{Q}_{L^p}^{p-1}.
\end{align*}
Then we have
\begin{equation}\label{vaat}
	\begin{aligned}
	&\frac{d}{dt}\norm{Q}_{L_{\eta}^p(\mathbb R))}^p+C_p\int_{\mathbb R}v_{\eta}|Q|^pd\eta +C_{p,K_1}\int_{\mathbb R}|Q|^{p-2}Q_{\eta}^2d\eta \\&\leq C\ln^{2}(3+t)(1+t)^{-2+\frac{1}{p}}\norm{Q}_{L^p}^{p-2}+C(1+t)^{-\frac{3}{2}+\frac{1}{2p}}\ln^{2}(3+t)\norm{Q}_{L^p}^{p-1}.
\end{aligned}
\end{equation}
Multiplying \eqref{vaat} by $(1+t)^{\alpha}$ with some sufficiently large positive constant $\alpha$ and
integrating the resulting equation from $0$ to $t$, and noting that $v_{\xi\xi}|_{t=0}=0$,
we can get
\begin{equation}\label{QQQ}
\begin{aligned}
	&(1+t)^\alpha\norm{Q}_{L_{\eta}^p(\mathbb R))}^p+C_p\int_{0}^{t}\int_{\mathbb R}(1+\tau)^\alpha v_{\eta}|Q|^pd\eta d\tau+C_{p,K_1}\int_{0}^{t}(1+\tau)^\alpha\norm{\partial_{\eta} (|Q|^{\frac{p}{2}})}_{L^2}^2d\tau\\
	&\leq C\int_{0}^{t}(1+\tau)^{\alpha-1}\norm{Q}_{L^p}^p(\tau)d\tau+C\ln^2(3+t)\int_{0}^{t}(1+\tau)^{\alpha-2+\frac{1}{p}}\norm{Q}_{L^p}^{p-2}(\tau)d\tau\\&~~~+
	C\ln^2(3+t)\int_{0}^{t}(1+\tau)^{\alpha-\frac{3}{2}+\frac{1}{2p}}\norm{Q}_{L^p}^{p-1}(\tau)d\tau\\
	&\triangleq J_5+J_6+J_7.
\end{aligned}
\end{equation}
By \eqref{esj1} and \eqref{vaa1e}, we have
\begin{equation}\label{QQQ1}
	\begin{aligned}
J_5&=CI(1,0,1)\\&\leq \epsilon\int_{0}^{t}(1+\tau)^{\alpha}\norm{\partial_{\eta} (|Q|^{\frac{p}{2}})}_{L_{\eta}^2(\mathbb R)}^2(\tau)d\tau+C_{\epsilon}\int_{0}^{t}(1+\tau)^{\alpha-\frac{p+1}{2}}\norm{Q}_{L_{\eta}^1(\mathbb R)}^pd\tau\\
&\leq \epsilon\int_{0}^{t}(1+\tau)^{\alpha}\norm{\partial_{\eta} (|Q|^{\frac{p}{2}})}_{L_{\eta}^2(\mathbb R)}^2(\tau)d\tau+C_{\epsilon}(1+t)^{\alpha-\frac{p}{2}+\frac{1}{2}}\ln^{p}(1+t),
\end{aligned}
\end{equation}
\begin{equation}\label{QQQ2}
	\begin{aligned}
		J_6&=C\ln^2(3+t)I(2-\frac{1}{p},2,1)\\&\leq \epsilon\int_{0}^{t}(1+\tau)^{\alpha}\norm{\partial_{\eta} (|Q|^{\frac{p}{2}})}_{L_{\eta}^2(\mathbb R)}^2(\tau)d\tau+C_{\epsilon}\ln^{2p}(3+t)\int_{0}^{t}(1+\tau)^{\alpha-\frac{p+1}{2}}\norm{Q}_{L_{\eta}^1(\mathbb R)}^{\frac{p(p-2)}{2p-1}}(\tau)d\tau\\
		&\leq \epsilon\int_{0}^{t}(1+\tau)^{\alpha}\norm{\partial_{\eta} (|Q|^{\frac{p}{2}})}_{L_{\eta}^2(\mathbb R)}^2(\tau)d\tau+C_{\epsilon}(1+t)^{\alpha-\frac{p+1}{2}+1}\ln^{\frac{p(p-2)}{2p-1}+2p}(3+t)\\
		&\leq \epsilon\int_{0}^{t}(1+\tau)^{\alpha}\norm{\partial_{\eta} (|Q|^{\frac{p}{2}})}_{L_{\eta}^2(\mathbb R)}^2(\tau)d\tau+C_{\epsilon}(1+t)^{\alpha-\frac{p}{2}+\frac{1}{2}}\ln^{3p}(3+t),
	\end{aligned}
\end{equation}
and
\begin{equation}\label{QQQ3}
	\begin{aligned}
		J_7&=C\ln^2(3+t)I(\frac{3}{2}-\frac{1}{2p},1,1)\\&\leq \epsilon\int_{0}^{t}(1+\tau)^{\alpha}\norm{\partial_{\eta} (|Q|^{\frac{p}{2}})}_{L_{\eta}^2(\mathbb R)}^2(\tau)d\tau+C_{\epsilon,p}\ln^{2p}(3+t)\int_{0}^{t}(1+\tau)^{\alpha-\frac{p+1}{2}}\norm{Q}_{L_{\eta}^1(\mathbb R)}^{\frac{2p(p-1)}{3p-1}}(\tau)d\tau\\
		&\leq \epsilon\int_{0}^{t}(1+\tau)^{\alpha}\norm{\partial_{\eta} (|Q|^{\frac{p}{2}})}_{L_{\eta}^2(\mathbb R)}^2(\tau)d\tau+C_{\epsilon}(1+t)^{\alpha-\frac{p+1}{2}+1}\ln^{\frac{2p(p-1)}{3p-1}+2p}(3+t)\\
		&\leq \epsilon\int_{0}^{t}(1+\tau)^{\alpha}\norm{\partial_{\eta} (|Q|^{\frac{p}{2}})}_{L_{\eta}^2(\mathbb R)}^2(\tau)d\tau+C_{\epsilon}(1+t)^{\alpha-\frac{p}{2}+\frac{1}{2}}\ln^{3p}(3+t).
	\end{aligned}
\end{equation}
Substituting the estimates \eqref{QQQ1}-\eqref{QQQ3} into \eqref{QQQ} and taking $\epsilon$ suitably small, we have
\begin{align*}
	&(1+t)^\alpha\norm{Q}_{L_{\eta}^p(\mathbb R))}^p\leq C(1+t)^{\alpha-\frac{p}{2}+\frac{1}{2}}\ln^{3p}(3+t).
\end{align*}
Therefore, we can get
\begin{align*}
\norm{Q}_{L_{\eta}^p(\mathbb R))}^p\leq C(1+t)^{-\frac{p}{2}+\frac{1}{2}}\ln^{3p}(3+t),~p\geq 2.
\end{align*}
Similar to \eqref{lp2}, we have
\begin{align}\label{vabr2}
	\norm{Q}_{L^p(\mathbb R^2)}^p\leq C(1+t)^{-\frac{p}{2}+\frac{1}{2}}\ln^{3p}(3+t),
\end{align}
which prove \eqref{vaap}. 
Then the proof of \cref{vl2} is completed.
\end{proof}

\subsection{Decay Estimates for the Perturbation}
Set the perturbation of the solution $u(t,\xi,\eta)$ to \eqref{equ} around the viscous profile $v(t,\xi,\eta)$ in \eqref{v1} by
\begin{equation*}
	\phi(t,\xi,\eta):=u(t,\xi,\eta)-v(t,\xi,\eta).
\end{equation*}
Then the perturbation $\phi(t,\xi,\eta)$ satisfies
\begin{equation}\label{eqphi}
	\phi_t+\phi\phi_{\eta} +v_{\eta}\phi+v\phi_{\eta}=K(\xi)\phi_{\eta\eta }+A(\xi)\phi_{\xi\eta}+2\phi_{\xi\xi}+B(\xi)\phi_{\eta}+A(\xi)v_{\xi\eta}+2v_{\xi\xi}+B(\xi)v_{\eta},
\end{equation}
with the initial data
\begin{equation}\label{phi0}
	\phi(0,\xi,\eta)= \phi_0(
	\xi,\eta):=u_0(\xi,\eta)-w_0(\eta).
\end{equation}

We have the following proposition for the perturbation $\phi(t,\xi,\eta)$.
\begin{Prop}\label{decayphi}
	If $\phi_0(\xi,\eta)\in H^1(\mathbb{R}^2)\cap L^\infty(\mathbb{R}^2)$, then the Cauchy problem \eqref{eqphi}-\eqref{phi0} has a unique global solution $\phi\in C([0,+\infty);~H^1(\mathbb{R}^2))$.
	Moreover, 
	it holds that
	\begin{align}\label{phi5}
		&\di \|{\phi}	\|_{L^6(\mathbb R^2)}\leq 	\|{\phi_0}	\|_{L^6(\mathbb R^2)}+C\ln (3+t),\\
\label{phip}
		&\di \|{\phi}	\|_{L^p(\mathbb R^2)}\leq C(1+t)^{-\frac{1}{8} +\frac{1}{p}}\ln^3(3+t),~6< p<+\infty,\\
	\label{phibp}
		&\di \|{\phi_{\eta}}	\|_{L^p(\mathbb R^2)}\leq C_{\delta}(1+t)^{-\frac{1}{4}+\delta +\frac{7}{4p}}\ln^7(3+t),~~6<p<+\infty,\\
		\label{phiap}
		&\di \|{\phi_{\xi}}	\|_{L^p(\mathbb R^2)}\leq C(1+t)^{-\frac{1}{8}+\frac{2}{p}}\ln^3(3+t),~6< p<+\infty,
	\end{align}
	where $\delta>0$ is an arbitrarily small positive constant, and $C_{\delta}$ is a positive constant depending on $\delta$.
\end{Prop}
\begin{proof}
Since the local existence and uniqueness of the classical solution to  \eqref{eqphi} is standard, we omit its proof for brevity. Thus it suffices to prove the estimates \eqref{phi5}-\eqref{phibp}.

\noindent$\mathbf{Step~1.~}$ To prove the $L^{6}$ estimate \eqref{phi5},
we multiply \eqref{eqphi} by $|\phi|^{p-2}\phi$ and
integrate it over $\mathbb R^2$ with respect to $\xi$ and $\eta$.
We have
\begin{equation}\label{inevb}
	\begin{aligned}
	&\frac{1}{p}\frac{d}{dt}\norm{\phi}_{L^p(R^2)}^p+\frac{p-1}{p}\iint_{\mathbb R^2}v_{\eta}|\phi|^pd\xi d\eta +(p-1)\iint_{\mathbb R^2}|\phi|^{p-2}(K(\xi)\phi_{\eta}^2+A(\xi)\phi_{\xi}\phi_{\eta}+2\phi_{\xi}^2)d\xi d\eta \\&=
	\iint_{\mathbb R^2}A(\xi)v_{\xi\eta}|\phi|^{p-2}\phi d\xi d\eta +2\iint_{\mathbb R^2}v_{\xi\xi}|\phi|^{p-2}\phi d\xi d\eta +\iint_{\mathbb R^2}B(\xi)v_{\eta}|\phi|^{p-2}\phi d\xi d\eta \\
	&\triangleq J_1+J_2+J_3.
\end{aligned}
\end{equation}
The terms $J_1$ and $J_2$ can be estimated as 
\begin{align}\label{vbj1}
	J_1\leq C\iint_{\mathbb R^2}v_{\xi}|\phi|^{p-2}\phi_{\eta} d\xi d\eta \leq
	\epsilon \iint_{\mathbb R^2}|\phi|^{p-2}\phi_{\eta}^2d\xi d\eta +C_{\epsilon}\iint_{\mathbb R^2}v_{\xi}^2|\phi|^{p-2}d\xi d\eta ,
\end{align}
\begin{align}\label{vbj2}
	J_2\leq C\iint_{\mathbb R^2}v_{\xi}|\phi|^{p-2}\phi_{\xi} d\xi d\eta \leq
	\epsilon \iint_{\mathbb R^2}|\phi|^{p-2}\phi_{\xi}^2d\xi d\eta+C_{\epsilon}\iint_{\mathbb R^2}v_{\xi}^2|\phi|^{p-2}d\xi d\eta .
\end{align}
For $J_3$, if $p=2$, since $B(\xi)=0$ for $\xi\notin [G_1,G_2]$, then
\begin{equation}\label{l2j3}
\begin{aligned}	J_3&\leq \epsilon\iint_{\mathbb R^2}v_{\eta}|\phi|^2d\xi d\eta +C_{\epsilon}\iint_{ D}|B(\xi)|v_{\eta}d\xi d\eta\\ 
	&\leq \epsilon\iint_{\mathbb R^2}v_{\eta}|\phi|^2d\xi d\eta +C_\epsilon,
	\end{aligned}
\end{equation}
where the region $D$ is given by \eqref{sva}. Then combining \eqref{para}, \eqref{inevb} and \eqref{vbj1}-\eqref{l2j3}, we can get
\begin{align*}
	\frac{d}{dt}\norm{\phi}_{L^2(R^2)}^2\leq C\norm{ v_\xi}_{L^2(\mathbb R^2)}^2+C\leq C(1+t)^{-\frac{1}{2}}\ln^2(1+t)+C,
\end{align*}
it follows that
\begin{align}\label{esmate2}
	\norm{\phi}_{L^2(\mathbb R^2)}^2\leq C\norm{\phi_0}_{L^2(\mathbb R^2)}^2+C(1+t).
\end{align}
If $p>2$, by \eqref{Ga},  we have 
\begin{equation*}
	B(\xi)=\frac{-2\varphi''(G(\xi))}{(1-\varphi'(G(\xi)))^3}=\partial_{\xi}\frac{-2}{1-\varphi'(G(\xi))}.
\end{equation*}
Then it holds that
\begin{equation}\label{g2j3}
	\begin{aligned}
	J_3&=\iint_{\mathbb R^2}\frac{2}{1-\varphi'(G(\xi))}v_{\xi\eta}|\phi|^{p-2}\phi d\xi d\eta +\iint_{\mathbb R^2}\frac{2(p-1)}{1-\varphi'(G(\xi))}v_{\eta}|\phi|^{p-2}\phi_{\xi} d\xi d\eta \\
	&\leq
	\epsilon \iint_{\mathbb R^2}|\phi|^{p-2}\phi_{\eta}^2d\xi d\eta +\epsilon \iint_{\mathbb R^2}|\phi|^{p-2}\phi_{\xi}^2d\xi d\eta +
	C_{\epsilon}\iint_{\mathbb R^2}(v_{\xi}^2+v_{\eta}^2)|\phi|^{p-2}d\xi d\eta .
\end{aligned}\end{equation}
Therefore, when $p>2$, combining \eqref{para}, \eqref{vbj1}, \eqref{vbj2} and \eqref{g2j3}, we can obtain
\begin{equation}\label{esmatep}
\begin{aligned}
	&\frac{d}{dt}\norm{\phi}_{L^p(\mathbb R^2)}^p+C_p\iint_{\mathbb R^2}v_{\eta}|\phi|^pd\xi d\eta +C_{p,d}\norm{\nabla(|\phi|^{\frac{p}{2}})}_{L^2(\mathbb R^2)}^2\\\leq&
	C\iint_{\mathbb R^2}v_{\xi}^2|\phi|^{p-2}d\xi d\eta +C\iint_{\mathbb R^2}v_{\eta}^2|\phi|^{p-2}d\xi d\eta .
\end{aligned}
\end{equation}
Furthermore, 
when $p> 4$, since 
\begin{align*}
	v_{\xi}^2(t,\xi,\eta)=\partial_{\xi}(\int_{G_1}^{\xi}v_{\xi}^2(t,\alpha,\eta)d\alpha),
\end{align*}
and by \eqref{sva}, we have
 \begin{align*}
	\int_{G_1}^{\xi}v_{\xi}^2(t,\alpha,\eta )d\alpha\leq \int_{G_1}^{G_2}v_{\xi}^2(t,\alpha,\eta)d\alpha,~\xi\in \mathbb{R}.
\end{align*}
 Then, by integration by parts and \eqref{vainf}, we obtain 
\begin{equation}\label{n1}
	\begin{aligned}
	\iint_{\mathbb R^2}v_{\xi}^2|\phi|^{p-2}d\xi d\eta &=-(p-2)\iint_{\mathbb R^2}(\int_{G_1}^{\xi}v_{\xi}^2(t,\alpha,\eta)d\alpha)|\phi|^{p-4}\phi \phi_{\xi}d\xi d\eta \\
	&\leq \epsilon \iint_{\mathbb R^2}|\phi|^{p-4}\phi^2\phi_{\xi}^2d\xi d\eta +C_{\epsilon}\iint_{\mathbb R^2}(\int_{G_1}^{G_2}v_{\xi}^2(t,\alpha,\eta)d\alpha)^2|\phi|^{p-4}d\xi d\eta \\
	&\leq \epsilon \iint_{\mathbb R^2}|\phi|^{p-2}\phi_{\xi}^2d\xi d\eta +C_{\epsilon}(G_2-G_1)^2\norm{v_{\xi}}_{L^{\infty}(\mathbb R^2)}^4\iint_{\mathbb R^2}|\phi|^{p-4}d\xi d\eta \\
	&\leq \epsilon \iint_{\mathbb R^2}|\phi|^{p-2}\phi_{\xi}^2d\xi d\eta +C_{\epsilon}(1+t)^{-2}\ln^4(3+t)\norm{\phi}_{L^{p-4}}^{p-4},
\end{aligned}
\end{equation}
From \eqref{vbp}, it holds that
\begin{equation}\label{n2}
	\begin{aligned}
 \iint_{\mathbb R^2}v_{\eta}^2|\phi|^{p-2}d\xi d\eta\leq \norm{v_{\eta}}_{L^{\infty}(\mathbb R^2)}^2\norm{\phi}_{L^{p-2}}^{p-2}\leq C(1+t)^{-2}\norm{\phi}_{L^{p-2}}^{p-2}.
	\end{aligned}
\end{equation}
Substituting \eqref{n1} and \eqref{n2} into \eqref{esmatep} and taking $\epsilon$ suitably small, we can get
\begin{equation}\label{esmatep1}
	\begin{aligned}
	&\frac{d}{dt}\norm{\phi}_{L^p(\mathbb R^2)}^p+C_p\iint_{\mathbb R^2}v_{\eta}|\phi|^pd\xi d\eta +C_{p,d}\norm{\nabla(|\phi|^{\frac{p}{2}})}_{L^2(\mathbb R^2)}^2\\
	&\leq
	C(1+t)^{-2}\ln^4(3+t)\norm{\phi}_{L^{p-4}}^{p-4}+C(1+t)^{-2}\norm{\phi}_{L^{p-2}}^{p-2}, ~p>4.
\end{aligned}
\end{equation}
Thus choosing $p=6$ in \eqref{esmatep1}, we have
\begin{align}\label{esmate624}
	\frac{d}{dt}\norm{\phi}_{L^6(\mathbb R^2)}^6
	\leq
	C(1+t)^{-2}\ln^4(3+t)\norm{\phi}_{L^{2}}^2+C(1+t)^{-2}\norm{\phi}_{L^{4}}^{4}.
\end{align}

We next estimate $\norm{\phi}_{L^4(\mathbb{R}^2)}$. According to \eqref{esmatep},
we have
\begin{align*}	\frac{d}{dt}\norm{\phi}_{L^4(\mathbb R^2)}^4&\leq
	C\|{v_{\xi}}\|_{L^{4}(\mathbb{R}^2)}^2\|{\phi}\|_{L^4(\mathbb R^2)}^2+C\|{v_{\eta}}\|_{L^{\infty}(\mathbb{R}^2)}^2\|{\phi}\|_{L^2(\mathbb R^2)}^2\\
&\leq C(1+t)^{-\frac{3}{4}}\ln^2(1+t)\|{\phi}\|_{L^4(\mathbb R^2)}^2+C(1+t)^{-1},
\end{align*}
it follows that
\begin{align}\label{esmate3}
		\norm{\phi}_{L^4(\mathbb R^2)}^4\leq \|{\phi_0}\|_{L^4(\mathbb R^2)}^4
		+C(1+t).
\end{align}

Then combining \eqref{esmate624}, the $L^2$-estimate \eqref{esmate2} and  $L^4$-estimate\eqref{esmate3},  we can get
\begin{align}\label{esmate6}
	\frac{d}{dt}\norm{\phi}_{L^6(\mathbb R^2)}^6
	\leq
C(1+t)^{-1}\ln^4(3+t),
\end{align}
thus 
\begin{align}\label{esmate6e}
	\norm{\phi}_{L^6(\mathbb R^2)}^6
	\leq \norm{\phi_0}_{L^6(\mathbb R^2)}^6+
	C\ln^5(3+t),
\end{align}
we complete the proof of \eqref{phi5}.

\noindent$\mathbf{Step~2.~}$ We prove the $L^{p}$-estimate \eqref{phip} for $\phi$. From \eqref{vbj1} and \eqref{vbj2}, and note that $v_{\eta}\notin L^p(\mathbb R^2),~1\leq  p<\infty$, we calculate \eqref{inevb} as  
\begin{equation}\label{etp1}
	\begin{aligned}
	&\frac{d}{dt}\norm{\phi}_{L^p(\mathbb R^2)}^p+C_p\iint_{\mathbb R^2}v_{\eta}|\phi|^pd\xi d\eta +C_{p,d}\norm{\nabla(|\phi|^{\frac{p}{2}})}_{L^2(\mathbb R^2)}^2\\
	&\leq
	C\norm{v_{\xi}}_{L^{p}}^2\norm{\phi}_{L^{p}}^{p-2}+C\norm{v_{\eta}}_{L^{p}(D)}\norm{\phi}_{L^{p}}^{p-1}\\
	&\leq
	C(1+t)^{-1+\frac{1}{p}}\ln^{2}(1+t)\norm{\phi}_{L^{p}}^{p-2}+C(1+t)^{-1+\frac{1}{p}}\norm{\phi}_{L^{p}}^{p-1}.
\end{aligned}
\end{equation}
By multiplying \eqref{etp1} by $(1+t)^{\alpha}$, where $\alpha$ is some sufficiently large positive constant, and then integrating the resulting equation with respect to $t$, we have 
\begin{equation}\label{tphie}
	\begin{aligned}
	&\quad(1+t)^\alpha\norm{\phi}_{L^p(\mathbb R^2)}^p+C_p\int_{0}^{t}\int_{\mathbb R^2}(1+\tau)^\alpha v_{\eta}|\phi|^pd\xi d\eta d\tau+C_d\int_{0}^{t}(1+\tau)^\alpha\norm{\nabla (|\phi|^{\frac{p}{2}})}_{L^2}^2d\tau\\
	&\leq\norm{\phi_0}_{L^p(\mathbb R^2)}^p+ \int_{0}^{t}(1+\tau)^{\alpha-1}\norm{\phi}_{L^{p}}^{p}(\tau)d\tau+C\ln^2(1+t)\int_{0}^{t}(1+\tau)^{\alpha-1+\frac{1}{p}}\norm{\phi}_{L^p}^{p-2}d\tau\\
	&\quad+C\int_{0}^{t}(1+\tau)^{\alpha-1+\frac{1}{p}}\norm{\phi}_{L^p}^{p-1}d\tau\\
	&\triangleq \norm{\phi_0}_{L^p(\mathbb R^2)}^p +J_1+J_2+J_3.
\end{aligned}
\end{equation}
Applying inequality \eqref{esj3}
and the estimate \eqref{esmate6e}, we have
\begin{equation}\label{It1}
	\begin{aligned}
	J_1&=I(1,0,2)\\&\leq\epsilon\int_{0}^{t}(1+\tau)^{\alpha}\norm{\nabla (|\phi|^{\frac{p}{2}})}_{L^2(\mathbb{R}^2)}^2(\tau)d\tau+C_{\epsilon}\int_{0}^{t}(1+\tau)^{\alpha-\frac{p}{6}}\norm{\phi}_{L^6(\mathbb{R}^2)}^p(\tau)d\tau\\
		&\leq\epsilon\int_{0}^{t}(1+\tau)^{\alpha}\norm{\nabla (|\phi|^{\frac{p}{2}})}_{L^2(\mathbb{R}^2)}^2(\tau)d\tau+C_{\epsilon}(1+t)^{\alpha-\frac{p}{6}+1}\ln^p(3+t),
\end{aligned}
\end{equation}
\begin{equation}\label{It2}
	\begin{aligned}
	J_2&=C\ln^2(1+t)I(1-\frac{1}{p},2,2)\\&\leq\epsilon\int_{0}^{t}(1+\tau)^{\alpha}\norm{\nabla (|\phi|^{\frac{p}{2}})}_{L^2(\mathbb{R}^2)}^2(\tau)d\tau+C_{\epsilon}\ln^{2p}(1+t)\int_{0}^{t}(1+\tau)^{\alpha-\frac{p(p-1)}{8p-12}}\norm{\phi}_{L^6(\mathbb{R}^2)}^{\frac{6p(p-2)}{8p-12}}(\tau)d\tau\\
	&\leq\epsilon\int_{0}^{t}(1+\tau)^{\alpha}\norm{\nabla (|\phi|^{\frac{p}{2}})}_{L^2(\mathbb{R}^2)}^2(\tau)d\tau+C_{\epsilon}(1+t)^{\alpha-\frac{p}{8}+1}\ln^{3p}(3+t),
\end{aligned}
\end{equation}
\begin{equation}\label{It3}
	\begin{aligned}
	J_3&=CI(1-\frac{1}{p},1,2)\\&\leq\epsilon\int_{0}^{t}(1+\tau)^{\alpha}\norm{\nabla (|\phi|^{\frac{p}{2}})}_{L^2(\mathbb{R}^2)}^2(\tau)d\tau+C_{\epsilon}\int_{0}^{t}(1+\tau)^{\alpha-\frac{p(p-1)}{7p-6}}\norm{\phi}_{L^6(\mathbb{R}^2)}^{\frac{6p(p-1)}{7p-6}}(\tau)d\tau\\
	&\leq\epsilon\int_{0}^{t}(1+\tau)^{\alpha}\norm{\nabla (|\phi|^{\frac{p}{2}})}_{L^2(\mathbb{R}^2)}^2(\tau)d\tau+C_{\epsilon}(1+t)^{\alpha-\frac{p}{8}+1}\ln^{p}(3+t).
\end{aligned}
\end{equation}
Substituting  the estimates \eqref{It1}-\eqref{It3} into \eqref{tphie} and taking $\epsilon$ suitably small, we have
\begin{align*}
	&(1+t)^\alpha\norm{\phi}_{L^p(\mathbb R^2)}^p+C_p\int_{0}^{t}\int_{\mathbb R^2}(1+\tau)^\alpha v_{\eta}|\phi|^pd\xi d\eta d\tau+C_d\int_{0}^{t}(1+\tau)^\alpha\norm{\nabla (|\phi|^{\frac{p}{2}})}_{L^2}^2(\tau)d\tau\\
	&\leq\norm{\phi_0}_{L^p(\mathbb R^2)}^p+C(1+t)^{\alpha-\frac{p}{8}+1}\ln^{3p}(3+t).
\end{align*}
Then we can get
\begin{align}\label{epphi}
\norm{\phi}_{L^p(\mathbb R^2)}^p\leq C(1+t)^{-\frac{p}{8}+1}\ln^{3p}(3+t),~p>6,
\end{align}
thus we reach the desired estimate \eqref{phip}.

\noindent$\mathbf{Step~3.~}$ To prove the $L^p$-estimate $\eqref{phibp}$ for $\phi_{\eta}$, we differentiate \eqref{eqphi} with respect to $\eta$ and put $W:=\phi_{\eta}$. Then we have
\begin{equation}\label{ntphib}
\begin{aligned}
		W_t+\phi W_{\eta} +W^2+v_{\eta\eta}\phi+2v_{\eta}W+vW_{\eta}=&K(\xi)W_{\eta\eta }+A(\xi)W_{\xi\eta}+2W_{\xi\xi}+B(\xi)W_{\eta}\\&+A(\xi)v_{\xi\eta\eta}+2v_{\xi\xi\eta}+B(\xi)v_{\eta\eta}.
\end{aligned}
\end{equation}
Multiplying \eqref{ntphib} by $|W|^{p-2}W$ and
integrating it over $\mathbb R^2$, we obtain
\begin{equation}\label{M1}
	\begin{aligned}
	&\quad\frac{1}{p}\frac{d}{dt}\norm{W}_{L^p(R^2)}^p+\frac{p-1}{p}\iint_{\mathbb R^2}v_{\eta}|W|^pd\xi d\eta \\&\quad+(p-1)\iint_{\mathbb R^2}|W|^{p-2}(K(\xi)W_{\eta}^2+A(\xi)W_{\xi}W_{\eta}+2W_{\xi}^2)d\xi d\eta \\
	&=(p-1)\iint_{\mathbb R^2}\phi|W|^{p-2}WW_{\eta}d\xi d\eta +(p-1)\iint_{\mathbb R^2}v_{\eta}\phi|W|^{p-2}W_{\eta} d\xi d\eta \\&\quad-(p-1)\iint_{\mathbb R^2}A(\xi)v_{\xi\eta}|W|^{p-2}W_{\eta} d\xi d\eta 
	-(p-1)\iint_{\mathbb R^2}v_{\xi\eta}|W|^{p-2}W_{\xi} d\xi d\eta \\&\quad-(p-1)\iint_{\mathbb R^2}B(\xi)v_{\eta}|W|^{p-2}W_{\eta}d\xi d\eta \\
	&\triangleq W_1+W_2+W_3+W_4+W_5,
\end{aligned}
\end{equation}
By interpolation inequality \eqref{puin1} and estimate \eqref{epphi}, we have
\begin{equation}\label{Wj1}
	\begin{aligned}
		W_1&\leq C\iint_{\mathbb R^2}|\phi||W|^{p-1}|W_{\eta}|d\xi d\eta \\&\leq C\iint_{\mathbb R^2}|\phi||W|^{\frac{p}{2}}	|\nabla(|W|^{\frac{p}{2}})|d\xi d\eta \\
		&\leq C\norm{\phi}_{L^{k}(\mathbb R^2)}\norm{|W|^{\frac{p}{2}}}_{L^{\frac{2k}{k-2}}(\mathbb R^2)}\norm{ \nabla(|W|^{\frac{p}{2}})}_{L^2(\mathbb{R}^2)}\\
		&\leq C\norm{\phi}_{L^{k}(\mathbb R^2)}\norm{W}_{L^{p}(\mathbb R^2)}^{\frac{(k-2)p}{2k}}\norm{ \nabla(|W|^{\frac{p}{2}})}_{L^2(\mathbb{R}^2)}^{\frac{k+2}{k}}\\
		&\leq\epsilon \norm{ \nabla(|W|^{\frac{p}{2}})}_{L^2(\mathbb{R}^2)}^2+C_{\epsilon}\norm{\phi}_{L^k(\mathbb R^2)}^{\frac{2k}{k-2}}\norm{W}_{L^p(\mathbb R^2)}^p\\
		&\leq\epsilon \norm{ \nabla(|W|^{\frac{p}{2}})}_{L^2(\mathbb{R}^2)}^2+C_{\epsilon}\norm{\phi}_{L^k(\mathbb R^2)}^{\frac{2k}{k-2}}\norm{ \nabla(|W|^{\frac{p}{2}})}_{L^2(\mathbb{R}^2)}^{\frac{2p}{p+2}}\norm{\phi}_{L^p(\mathbb R^2)}^{\frac{2p}{p+2}}\\
		&\leq\epsilon_1 \norm{ \nabla(|W|^{\frac{p}{2}})}_{L^2(\mathbb{R}^2)}^2+C_{\epsilon_1}\norm{\phi}_{L^k(\mathbb R^2)}^{\frac{k(p+2)}{k-2}}\norm{\phi}_{L^p(\mathbb R^2)}^p\\
		&\leq \epsilon_1 \norm{ \nabla(|W|^{\frac{p}{2}})}_{L^2(\mathbb{R}^2)}^2+C_{\epsilon_1}(1+t)^{-\frac{(p+2)(k-8)}{8(k-2)}-\frac{p}{8}+1}\ln^{\frac{3k(p+2)}{k-2}+3p}(3+t)
\end{aligned}
\end{equation}
holds for any $k>6$, where we have used 
Gagliardo-Nirenberg inequality $$\norm{u}_{L^{\frac{2k}{k-2}}(\mathbb{R}^2)}\leq C\norm{u}_{L^{2}(\mathbb{R}^2)}^{1-\frac{2}{k}}\norm{\nabla u}_{L^2(\mathbb{R}^2)}^{\frac{2}{k}}.$$
 Then for an arbitrarily given small constant $\delta>0$, by choosing $k$ so large that $\frac{k-8}{k-2}>1-\frac{8p}{p+2}\delta$, we further have
\begin{equation}\label{Wj11}
	\begin{aligned}
	W_1 
	&\leq \epsilon \norm{ \nabla(|W|^{\frac{p}{2}})}_{L^2(\mathbb{R}^2)}^2+C_{\delta}(1+t)^{-\frac{(p+2)}{8}-\frac{p}{8}+\delta p+1}\ln^{7p}(3+t)\\
	&=\epsilon \norm{ \nabla(|W|^{\frac{p}{2}})}_{L^2(\mathbb{R}^2)}^2+C_{\delta}(1+t)^{-\frac{p}{4}+\delta p+\frac{3}{4}}\ln^{7p}(3+t).
\end{aligned}
\end{equation}
Next, it holds that
\begin{equation}\label{Wj2}
	\begin{aligned}
	W_2&\leq C\iint_{\mathbb R^2}v_{\eta}|\phi||W|^{\frac{p}{2}-1}|\nabla(|W|^{\frac{p}{2}})|d\xi d\eta  \\&\leq C\norm{v_{\eta}}_{L^{\infty}(\mathbb R^2)}\norm{\phi}_{L^{p}(\mathbb R^2)}\norm{W}_{L^{p}(\mathbb R^2)}^{\frac{p-2}{2}}\norm{ \nabla(|W|^{\frac{p}{2}})}_{L^2(\mathbb{R}^2)}\\
	&\leq C\norm{v_{\eta}}_{L^{\infty}(\mathbb R^2)}\norm{\phi}_{L^{p}(\mathbb R^2)}^{\frac{2p}{p+2}}\norm{ \nabla(|W|^{\frac{p}{2}})}_{L^2(\mathbb{R}^2)}^{\frac{2p}{p+2}}\\
	&\leq
	\epsilon \norm{ \nabla(|W|^{\frac{p}{2}})}_{L^2(\mathbb{R}^2)}^2+C_{\epsilon}\norm{v_{\eta}}_{L^{\infty}(\mathbb R^2)}^{\frac{p+2}{2}}\norm{\phi}_{L^{p}(\mathbb R^2)}^{p}\\
	&\leq \epsilon \norm{ \nabla(|W|^{\frac{p}{2}})}_{L^2(\mathbb{R}^2)}^2+C_{\epsilon}(1+t)^{-\frac{p}{4}+\frac{3}{4}}\ln^{3p}(3+t),
\end{aligned}
\end{equation}
and
\begin{equation}
	\begin{aligned}\label{Wj32}
	W_3 &\leq 
	\epsilon \iint_{\mathbb R^2}|W|^{p-2}W_{\eta}^2d\xi d\eta +C_{\epsilon}\iint_{\mathbb R^2}v_{\xi\eta}^2|W|^{p-2}d\xi d\eta \\
	&\leq \epsilon \norm{ \nabla(|W|^{\frac{p}{2}})}_{L^2(\mathbb{R}^2)}^2+C_{\epsilon}\norm{v_{\xi\eta}}_{L^{p}(\mathbb R^2)}^{2}\norm{W}_{L^{p}(\mathbb R^2)}^{p-2}\\
	&\leq \epsilon \norm{ \nabla(|W|^{\frac{p}{2}})}_{L^2(\mathbb{R}^2)}^2+
	C_{\epsilon}\norm{v_{\xi\eta}}_{L^{p}(\mathbb R^2)}^{2}\norm{ \nabla(|W|^{\frac{p}{2}})}_{L^2(\mathbb{R}^2)}^{\frac{2(p-2)}{p+2}}\norm{\phi}_{L^p(\mathbb R^2)}^{\frac{2(p-2)}{p+2}}\\
	&\leq \epsilon \norm{ \nabla(|W|^{\frac{p}{2}})}_{L^2(\mathbb{R}^2)}^2+
	C_{\epsilon}\norm{v_{\xi\eta}}_{L^{p}(\mathbb R^2)}^{\frac{p+2}{2}}\norm{\phi}_{L^p(\mathbb R^2)}^{\frac{p-2}{2}}\\
	&\leq \epsilon \norm{ \nabla(|W|^{\frac{p}{2}})}_{L^2(\mathbb{R}^2)}^2+C_{\epsilon}(1+t)^{-\frac{p}{4}+\frac{3}{4}}\ln^{3p}(3+t).
\end{aligned}
\end{equation}
For $W_4$, one has
\begin{align}\label{j4toj3}
	W_4\leq 
	\epsilon \iint_{\mathbb R^2}|W|^{p-2}W_{\xi}^2d\xi d\eta +C_{\epsilon}\iint_{\mathbb R^2}v_{\xi\eta}^2|W|^{p-2}d\xi d\eta ,
\end{align}
and we have obtained the estimate of second term on the right-hand side of \eqref{j4toj3} in \eqref{Wj32}.

For $W_5$, similar to the estimate \eqref{Wj32}, we have
\begin{align*}
	W_5 &\leq 
	\epsilon \iint_{\mathbb R^2}|W|^{p-2}W_{\eta}^2d\xi d\eta +C_{\epsilon}\iint_{\mathbb D}v_{\eta}^2|W|^{p-2}d\xi d\eta \\
	&\leq \epsilon \norm{ \nabla(|W|^{\frac{p}{2}})}_{L^2(\mathbb{R}^2)}^2+C_{\epsilon}\norm{v_{\eta}}_{L^{p}(\mathbb D)}^{2}\norm{W}_{L^{p}(\mathbb R^2)}^{p-2}\\
	&\leq \epsilon \norm{ \nabla(|W|^{\frac{p}{2}})}_{L^2(\mathbb{R}^2)}^2+C_{\epsilon}(1+t)^{-\frac{p}{4}+\frac{3}{4}}\ln^{3p}(3+t).
\end{align*}
Substituting the above estimates $W_1$-$W_5$ into \eqref{M1} and taking $\epsilon$ suitably small, we can obtain
\begin{align}\label{wet}
	\frac{d}{dt}\norm{W}_{L^p(R^2)}^p+C_p\iint_{\mathbb R^2}v_{\eta}|W|^pd\xi d\eta +C_{K_1,d}\norm{ \nabla(|W|^{\frac{p}{2}})}_{L^2(\mathbb{R}^2)}^2
	\leq C_{\delta}(1+t)^{-\frac{p}{4}+\delta p+\frac{3}{4}}\ln^{7p}(3+t).
\end{align}
Multiplying \eqref{wet} by $(1+t)^\alpha$ with some sufficiently large positive constant $\alpha$ and 
integrating the resulting equation from $0$ to $t$, we have
\begin{equation}\label{tphibe}
	\begin{aligned}
	&\quad (1+t)^\alpha\norm{W}_{L^p}^p+C_p\int_{0}^{t}\int_{\mathbb R^2}(1+\tau)^\alpha v_{\eta}|W|^pd\xi d\eta d\tau\\
	&\quad+C_d\int_{0}^{t}(1+\tau)^\alpha\norm{\nabla (|W|^{\frac{p}{2}})}_{L^2}^2(\tau)d\tau\\
	&\leq \norm{W_0}_{L^p}^p+ C\int_{0}^{t}(1+\tau)^{\alpha-1}\norm{W}_{L^{p}}^{p}(\tau)d\tau+C_{\delta}(1+t)^{\alpha-\frac{p}{4}+\delta p+\frac{7}{4}}\ln^{7p}(3+t)\\
	&\triangleq \norm{W_0}_{L^p}^p+J +C_{\delta}(1+t)^{\alpha-\frac{p}{4}+\delta p+\frac{7}{4}}\ln^{7p}(3+t).
\end{aligned}
\end{equation}
Using inequality \eqref{puin1}
and the estimate \eqref{epphi}, one has
\begin{equation}\label{phibtj}
	\begin{aligned}
	J&\leq C\int_{0}^{t}(1+\tau)^{\alpha-1}\norm{\phi}_{L^{p}}^{\frac{2p}{p+2}}\norm{\nabla (|W|^{\frac{p}{2}})}_{L^2}^{\frac{2p}{p+2}}(\tau)d\tau\\	
&\leq\epsilon\int_{0}^{t}(1+\tau)^{\alpha}\norm{\nabla (|W|^{\frac{p}{2}})}_{L^2}^2(\tau)d\tau+C_{\epsilon}\int_{0}^{t}(1+\tau)^{\alpha-\frac{p+2}{2}}\norm{\phi}_{L^p(\mathbb{R}^2)}^p(\tau)d\tau\\
	&\leq\epsilon\int_{0}^{t}(1+\tau)^{\alpha}\norm{\nabla (|W|^{\frac{p}{2}})}_{L^2}^2(\tau)d\tau+C_{\epsilon}(1+t)^{\alpha-\frac{5p}{8}+1}\ln^{3p}(3+t).
\end{aligned}
\end{equation}
Therefore, by \eqref{tphibe} and \eqref{phibtj}, we can obtain
\begin{align}\label{etpb}
	\norm{W}_{L^p(R^2)}^p
	\leq C_{\delta}(1+t)^{-\frac{p}{4}+\delta p+\frac{7}{4}}\ln^{7p}(3+t). 
\end{align}
Thus we have shown the estimate \eqref{phibp}.

\noindent$\mathbf{Step~4.~}$Finally, we prove the $L^p$-estimate $\eqref{phiap}$ for $\phi_{\xi}$. Differentiating \eqref{eqphi} with respect to $\xi$ and putting $Q:=\phi_{\xi}$, we have
\begin{equation}\label{ntphia}
	\begin{aligned}
	&Q_t+(\phi Q)_{\eta} +(v_{\xi}\phi)_{\eta}+v_{\eta}Q+vQ_{\eta}\\=&K(\xi)Q_{\eta\eta }+A(\xi)Q_{\xi\eta}+2Q_{\xi\xi}+B(\xi)Q_{\eta}
	+K'(\xi)\phi_{\eta\eta }+A'(\xi)Q_{\eta}+B'(\xi)\phi_{\eta}\\&+A(\xi)v_{\xi\xi\eta}+2v_{\xi\xi\xi}+B(\xi)v_{\xi\eta}+A'(\xi)v_{\xi\eta}+B'(\xi)v_{\eta}.
\end{aligned}
\end{equation}
Multiplying \eqref{ntphia} by $|Q|^{p-2}Q$ and
integrating it over $\mathbb R^2$, we obtain
\begin{equation}\label{Q1to7}
\begin{aligned}
	&\frac{1}{p}\frac{d}{dt}\norm{Q}_{L^p(\mathbb{R}^2)}^p+\frac{p-1}{p}\iint_{\mathbb R^2}v_{\eta}|Q|^pd\xi d\eta \\&+(p-1)\iint_{\mathbb R^2}|Q|^{p-2}(K(\xi)Q_{\eta}^2+A(\xi)Q_{\xi}Q_{\eta}+2Q_{\xi}^2)d\xi d\eta \\
	=&(p-1)\iint_{\mathbb R^2}\phi Q|Q|^{p-2}Q_{\eta}Wd\xi d\eta +(p-1)\iint_{\mathbb R^2}v_{\xi}\phi|Q|^{p-2}Q_{\eta} d\xi d\eta \\&-(p-1)\iint_{\mathbb R^2}K'(\xi){\phi}_{\eta}|Q|^{p-2}Q_{\eta} d\xi d\eta 
	-(p-1)\iint_{\mathbb R^2}B'(\xi){\phi}|Q|^{p-2}Q_{\eta} d\xi d\eta \\&-(p-1)\iint_{\mathbb R^2}v_{\xi\xi}|Q|^{p-2}(A(\xi)Q_{\eta}+2Q_{\xi})d\xi d\eta -(p-1)\iint_{\mathbb R^2}(B(\xi)+A'(\xi))v_{\xi}|Q|^{p-2}Q_{\eta}d\xi d\eta \\&+\iint_{\mathbb R^2}B'(\xi)v_{\eta}|Q|^{p-2}Qd\xi d\eta \\
	\triangleq &Q_1+Q_2+Q_3+Q_4+Q_5+Q_6+Q_7.
\end{aligned}
\end{equation}
For $Q_1$, similar to the estimate of $W_1$ in \eqref{Wj1}, we can get
\begin{align}\label{QQ1}
	Q_1\leq C\iint_{\mathbb R^2}|\phi||Q|^{p-1}|Q_{\eta}|d\xi d\eta &\leq  \epsilon \norm{ \nabla(|Q|^{\frac{p}{2}})}_{L^2(\mathbb{R}^2)}^2+C_{\epsilon,\delta}(1+t)^{-\frac{p}{4}+\delta p+\frac{3}{4}}\ln^{7p}(3+t).
\end{align}
For $Q_2$, the similar computations as in \eqref{Wj2} gives
\begin{equation}
	\begin{aligned}
	Q_2&\leq 
	\epsilon \norm{ \nabla(|Q|^{\frac{p}{2}})}_{L^2(\mathbb{R}^2)}^2+C_{\epsilon}\norm{v_{\xi}}_{L^{\infty}(\mathbb R^2)}^{\frac{p+2}{2}}\norm{\phi}_{L^{p}(\mathbb R^2)}^{p}\\
	&\leq \epsilon \norm{ \nabla(|Q|^{\frac{p}{2}})}_{L^2(\mathbb{R}^2)}^2+C_{\epsilon}(1+t)^{-\frac{3p}{8}+\frac{1}{2}}\ln^{4p}(3+t).
\end{aligned}
\end{equation}
Applying the interpolation inequality \eqref{puin1} for $Q=\phi_{\xi}$, we have

\begin{equation}
	\begin{aligned}
	Q_3&\leq \epsilon \norm{ \nabla(|Q|^{\frac{p}{2}})}_{L^2(\mathbb{R}^2)}^2+C_{\epsilon}\iint_{\mathbb R^2}{\phi}_{\eta}^2|Q|^{p-2}d\xi d\eta \\
	&\leq\epsilon \norm{ \nabla(|Q|^{\frac{p}{2}})}_{L^2(\mathbb{R}^2)}^2+C_{\epsilon}\norm{\phi_{\eta}}_{L^{p}(\mathbb R^2)}^2\norm{Q}_{L^{p}(\mathbb R^2)}^{p-2}\\
	&\leq\epsilon \norm{ \nabla(|Q|^{\frac{p}{2}})}_{L^2(\mathbb{R}^2)}^2+C_{\epsilon}\norm{\phi_{\eta}}_{L^{p}(\mathbb R^2)}^2\norm{\phi}_{L^{p}(\mathbb R^2)}^{\frac{2(p-2)}{p+2}}\norm{ \nabla(|Q|^{\frac{p}{2}})}_{L^2(\mathbb{R}^2)}^{\frac{2(p-2)}{p+2}}\\
	&\leq\epsilon_1 \norm{ \nabla(|Q|^{\frac{p}{2}})}_{L^2(\mathbb{R}^2)}^2+C_{\epsilon_1}\norm{\phi_{\eta}}_{L^{p}(\mathbb R^2)}^{\frac{p+2}{2}}\norm{\phi}_{L^{p}(\mathbb R^2)}^{\frac{p-2}{2}}\\
	&\leq \epsilon_1 \norm{ \nabla(|Q|^{\frac{p}{2}})}_{L^2(\mathbb{R}^2)}^2+C_{\delta}(1+t)^{-\frac{3p}{16}+\delta p+\frac{3}{2}}\ln^{6p}(3+t).
\end{aligned}
\end{equation}
The estimates of $Q_4$-$Q_6$ are similar to $Q_3$, we have
\begin{equation}
	\begin{aligned}
	Q_4&\leq 
	\epsilon_1 \norm{ \nabla(|Q|^{\frac{p}{2}})}_{L^2(\mathbb{R}^2)}^2+C_{\epsilon_1}\norm{\phi}_{L^{p}(\mathbb R^2)}^{\frac{p+2}{2}}\norm{\phi}_{L^{p}(\mathbb R^2)}^{\frac{p-2}{2}}\\
	&\leq \epsilon_1 \norm{ \nabla(|Q|^{\frac{p}{2}})}_{L^2(\mathbb{R}^2)}^2+C_{\epsilon_1}(1+t)^{-{\frac{p}{8}}+1}\ln^{3p}(3+t),
\end{aligned}
\end{equation}
\begin{equation}
	\begin{aligned}
	Q_5&\leq 
	\epsilon_1 \norm{ \nabla(|Q|^{\frac{p}{2}})}_{L^2(\mathbb{R}^2)}^2+C_{\epsilon_1}\norm{v_{\xi\xi}}_{L^{p}(\mathbb R^2)}^{\frac{p+2}{2}}\norm{\phi}_{L^{p}(\mathbb R^2)}^{\frac{p-2}{2}}\\
	&\leq \epsilon_1 \norm{ \nabla(|Q|^{\frac{p}{2}})}_{L^2(\mathbb{R}^2)}^2+C_{\epsilon_1}(1+t)^{-{\frac{p}{8}}+1}\ln^{3p}(3+t),
\end{aligned}
\end{equation}
\begin{equation}
	\begin{aligned}
	Q_6&\leq 
	\epsilon_1 \norm{ \nabla(|Q|^{\frac{p}{2}})}_{L^2(\mathbb{R}^2)}^2+C_{\epsilon_1}\norm{v_{\xi}}_{L^{p}(\mathbb R^2)}^{\frac{p+2}{2}}\norm{\phi}_{L^{p}(\mathbb R^2)}^{\frac{p-2}{2}}\\
	&\leq \epsilon_1 \norm{ \nabla(|Q|^{\frac{p}{2}})}_{L^2(\mathbb{R}^2)}^2+C_{\epsilon_1}(1+t)^{-{\frac{p}{8}}+1}\ln^{3p}(3+t).
\end{aligned}
\end{equation}
For $Q_7$, since $B'(\xi)=0$ for $\xi\notin[G_1,G_2]$, then it has 
\begin{equation}\label{QQ7}
	\begin{aligned}
	Q_7=\iint_{\mathbb D}B'(\xi)v_{\eta}|Q|^{p-2}Qd\xi d\eta &\leq C\norm{v_{\eta}}_{L^{p}(\mathbb D)}\norm{Q}_{L^{p}(\mathbb R^2)}^{p}\\
	&\leq C\norm{v_{\eta}}_{L^{p}(\mathbb D)}\norm{\phi}_{L^{p}(\mathbb R^2)}^{\frac{2p}{p+2}}\norm{ \nabla(|Q|^{\frac{p}{2}})}_{L^2(\mathbb{R}^2)}^{\frac{2p}{p+2}}\\
	&\leq \epsilon_1 \norm{ \nabla(|Q|^{\frac{p}{2}})}_{L^2(\mathbb{R}^2)}^2+C_{\epsilon_1}\norm{v_{\eta}}_{L^{p}(\mathbb D)}^{\frac{p+2}{2}}\norm{\phi}_{L^{p}(\mathbb R^2)}^{p}\\
	&\leq \epsilon_1 \norm{ \nabla(|Q|^{\frac{p}{2}})}_{L^2(\mathbb{R}^2)}^2+C_{\epsilon_1}(1+t)^{-{\frac{p}{8}}+1}\ln^{3p}(3+t).
\end{aligned}
\end{equation}
Collecting the estimates \eqref{QQ1}-\eqref{QQ7} and taking $\epsilon_1$ and $\delta$ suitably small, we finally get
\begin{align}\label{Qnte}
	&\frac{d}{dt}\norm{Q}_{L^p(R^2)}^p+C_p\iint_{\mathbb R^2}v_{\eta}|Q|^pd\xi d\eta +C_{K_1,d}\norm{ \nabla(|Q|^{\frac{p}{2}})}_{L^2(\mathbb{R}^2)}^2
	\leq C(1+t)^{-\frac{p}{8}+1}\ln^{3p}(3+t).
\end{align}
Multiplying \eqref{Qnte} by $(1+t)^\alpha$ with some sufficiently large positive constant $\alpha$ and 
integrating the resulting equation from $0$ to $t$, we have
\begin{equation*}
\begin{aligned}
		&(1+t)^\alpha\norm{Q}_{L^p}^p+C_p\int_{0}^{t}\int_{\mathbb R^2}(1+\tau)^\alpha v_{\eta}|Q|^pd\xi d\eta d\tau\\
	&\quad+C_d\int_{0}^{t}(1+\tau)^\alpha\norm{\nabla (|Q|^{\frac{p}{2}})}_{L^2}^2(\tau)d\tau\\
	&\leq \norm{Q_0}_{L^p}^p+ C\int_{0}^{t}(1+\tau)^{\alpha-1}\norm{Q}_{L^{p}}^{p}(\tau)d\tau+C(1+t)^{\alpha-\frac{p}{8}+2}\ln^{3p}(3+t)
	\\
	&  \triangleq \norm{Q_0}_{L^p}^p+\widetilde{J} +C(1+t)^{\alpha-\frac{p}{8}+2}\ln^{3p}(3+t).
\end{aligned}
\end{equation*}
Similar to the estimate in \eqref{phibtj}, we have
\begin{align}
	\widetilde{J}
	\leq\epsilon\int_{0}^{t}(1+\tau)^{\alpha}\norm{\nabla (|Q|^{\frac{p}{2}})}_{L^2(\mathbb{R}^2)}^2(\tau)d\tau+C_{\epsilon}(1+t)^{\alpha-\frac{5p}{8}+1}\ln^{3p}(3+t).
\end{align}
Therefore, we can obtain
\begin{align}\label{etpa}
	\norm{Q}_{L^p(R^2)}^p
	\leq C(1+t)^{-\frac{p}{8}+2}\ln^{3p}(3+t),
\end{align}
we arrive at the desired estimate \eqref{phiap}. The proof of \cref{decayphi}
is completed.
\end{proof}

\subsection{The proof of \cref{mt}}
We first prove \cref{mt} with the initial discontinuity $y=\varphi(x)$ replaced by $y=L_{\epsilon_0}(x)$ (See Figure 2). 
From Gagliardo-Nirenberg inequality and \cref{decayphi}, we have for any $p>6$ 
\begin{align}\norm{\phi}_{L^{\infty}(\mathbb R^2)}\leq C\norm{\phi}_{L^p(\mathbb R^2)}^{1-\frac{2}{p}}\norm{\nabla \phi}_{L^p(\mathbb R^2)}^{\frac{2}{p}}\leq C (1+t)^{-\frac{1}{8}+\frac{1}{p}+\frac{2}{p^2}}\ln^3(3+t),\end{align}
where $C$ is a positive constant depending on $p$.
 For an arbitrarily given small constant $\epsilon>0$, we can choose $p$ so large that $\frac{1}{p}+\frac{2}{p^2}\leq \epsilon$, then we have 
\begin{align}\label{phiinfty}
	\norm{\phi}_{L^{\infty}(\mathbb R^2)}\leq C_{\epsilon
	} (1+t)^{-\frac{1}{8}+\epsilon}\ln^3(3+t).
\end{align}
Combining the results \eqref{T1}, \eqref{T2} and \eqref{phiinfty}, we can obtain	
\begin{equation*}
	\begin{aligned}
	&\norm{u(t,x,y)-u^R(t,x,y)}_{L^{\infty}(\mathbb R^2)}\\\leq& \norm{u(t,\xi,\eta)-v(t,\xi,\eta)}_{L^{\infty}(\mathbb R^2)}+\norm{v(t,\xi,\eta)-w(t,\xi,\eta)}_{L^{\infty}(\mathbb R^2)}+\norm{w(t,\xi,\eta)-u^R(\frac{\eta}{t})}_{L^{\infty}(\mathbb R^2)}\\\leq&
	C_{\epsilon} (1+t)^{-\frac{1}{8}+\epsilon}\ln^3(3+t)+C(1+t)^{-\frac{1}{2}}\\
	\leq&
	C_{\epsilon} (1+t)^{-\frac{1}{8}+\epsilon}\ln^3(3+t),
\end{aligned}
\end{equation*}
where $u(t,x,y)$ is the solution to \eqref{eq}-\eqref{idata}, $u^R(t,x,y)$ is 2D non-self-similar rarefaction wave \eqref{Rs},  $w(t,\xi,\eta)$
 is the smooth approximation of self-similar planar rarefaction wave $u^R(\frac{\eta}{t})$, and $v(t,\xi,\eta)$ is the viscous profile as the solution to \eqref{v1}. Therefore, \cref{mt} is proved for the initial discontinuity $y=L_{\epsilon_0}(x)$.

Now, we prove that \cref{mt} holds for the initial discontinuity $y=\varphi(x)$ satisfying \cref{2DH} and \cref{cl}.

We denote $u^R_L(t,x,y)$ be the rarefaction wave solution to the Riemann problem \eqref{eq0}-\eqref{0data} with initial discontinuity $y-L_{\epsilon_0}(x)=0$, 
and $Z_L(x,y)$ is the implicit function determined by 
\begin{equation*}
	y-Z_L-L_{\epsilon_0}(x-Z_L)=0.
\end{equation*}

From the discussion in \eqref{vh0}-\eqref{uut}, we have 
\begin{equation}\label{1}\norm{w_0(Z(x,y))-w_0(Z_L(x,y))}_{H^1(\mathbb R^2)}<+\infty,\end{equation}
and \begin{equation}\label{2}\norm{u^R(t,x,y)-u^R_L(t,x,y)}_{L^{\infty}(\mathbb R^2)}\leq\frac{C}{t}.\end{equation} 
Since $u_0(x,y)-w_0(Z(x,y))\in H^1(\mathbb R^2)$, then
\begin{equation}\label{3}\norm{u_0(x,y)-w_0(Z_L(x,y))}_{H^1(\mathbb R^2)}<+\infty.\end{equation}
We have already proved that \cref{mt} holds for initial discontinuity $y=L_{\epsilon_0}(x)$, thus by \eqref{3}, it holds that
\begin{equation}\label{4}
		\norm{u(t,x,y)-u^R_L(t,x,y)}_{L^{\infty}(\mathbb R^2)}
		\leq
		C_{\epsilon} (1+t)^{-\frac{1}{8}+\epsilon}\ln^3(3+t).
\end{equation}
Combining \eqref{2} and \eqref{4}, we can obtain
\begin{equation*}
	\begin{aligned}
		&\norm{u(t,x,y)-u^R(t,x,y)}_{L^{\infty}(\mathbb R^2)}\\\leq& \norm{u(t,x,y)-u^R_L(t,x,y)}_{L^{\infty}(\mathbb R^2)}+\norm{u^R_L(t,x,y)-u^R(t,x,y)}_{L^{\infty}(\mathbb R^2)}\\\leq&
		C_{\epsilon} (1+t)^{-\frac{1}{8}+\epsilon}\ln^3(3+t)+\frac{C}{t}\\
		\leq&
		C_{\epsilon} (1+t)^{-\frac{1}{8}+\epsilon}\ln^3(3+t).
	\end{aligned}
\end{equation*}
Therefore, we complete the proof of \cref{mt}.


\section{Appendix}
In this appendix, we will prove  \cref{prop12} and \cref{Z12} in subsection 2.1.
Recalling that  $y=\varphi_1(x)$ and $y=\varphi_2(x)$ are two different initial discontinuities to Riemann problem \eqref{eq0}-\eqref{0data},  and $Z_i(x,y)~(i=1,2)$ are the implicit functions determined by 
\begin{equation}\label{zi}
	y-Z_i-\varphi_i (x-Z_i)=0,
\end{equation}
and $u^R_i(t,x,y)~(i=1,2)$ are 2D non-self-similar rarefaction wave defined in \eqref{Ris} to the Riemann problem \eqref{eq0}-\eqref{0data} corresponding to the initial discontinuity $y=\varphi_i(x)$.

Before proving \cref{prop12}, we first give several lemmas. The following lemma is a comparison Lemma.
\begin{Lem}\label{RRleq}
	Suppose both  $\varphi_1(x)$ and $\varphi_2(x)$ satisfy \cref{2DH}, and $\forall x\in \mathbb{R},  \varphi_2(x)\leq\varphi_1(x)$,  then
	\begin{equation}\label{uuleq}
		u^R_1(t,x,y)\leq u^R_2(t,x,y),~\forall t>0.
	\end{equation}
\end{Lem}
\begin{proof}
	We define the following six regions, for $i=1,2,$ 
\begin{equation}\label{region}
	\begin{array}{ll}
	\di \Omega_i^-:=\{(t,x,y)\in\mathbb R_+\times\mathbb R^2|~y<u_-t+\varphi_i(x-u_-t)\};\\[2mm]
	\Omega_i^c:=\{(t,x,y)\in\mathbb R_+\times\mathbb R^2|~u_-t+\varphi_i (x-u_-t)\leq y\leq u_+t+\varphi_i (x-u_+t)\};\\[2mm]
	\Omega_i^+:=\{(t,x,y)\in\mathbb R_+\times\mathbb R^2|~y>u_+t+\varphi_i (x-u_+t)\}.
	\end{array}
\end{equation}
	
	Then $\mathbb R_+\times\mathbb R^2=\Omega_i^-\cup\Omega_i^c\cup\Omega_i^+,~i=1,2$, and
	\begin{equation}\label{uri}
		u^R_i(t,x,y)=\begin{cases}
			u_-, &~ (t,x,y)\in \Omega_i^-,\\[2mm]
			\frac{Z_i(x,y)}{t},&~(t,x,y)\in \Omega_i^c, \\[1mm]
			u_+,&~(t,x,y)\in \Omega_i^+.
		\end{cases}
	\end{equation}
      Now we prove \cref{RRleq} by the following three cases:\\

	\underline{Case 1}. If $(t,x,y)\in \Omega_2^+$, then
	we have $u^R_1(t,x,y)\leq u_+=u^R_2(t,x,y)$.
	
	\
	
	\underline{Case 2}. If $(t,x,y)\in \Omega_2^c$:
	
	\quad \underline{Case 2.1}.  If $(t,x,y)\in \Omega_2^c\cap \Omega_1^+$, then
	\begin{align}	y>u_+t+\varphi_1(x-u_+t).
	\end{align}
	By $\varphi_2(x)\leq\varphi_1(x)$, we have
	\begin{align*}
		y -Z_2(x,y)-\varphi_2(x-Z_2(x,y))&>u_+t- Z_2(x,y)+\varphi _1(x-u_+t)-\varphi_2(x-Z_2(x,y))\\
		&\geq u_+t- Z_2(x,y)+\varphi _2(x-u_+t)-\varphi_2(x-Z_2(x,y))\\
		&=u_+t-Z_2(x,y)-\varphi'_2 (\theta)(u_+t-Z_2(x,y))\\
		&=(u_+t-Z_2(x,y))(1-\varphi'_2 (\theta))\geq 0.
	\end{align*}
	Then $(t,x,y)\notin \Omega_2^c$, which contradicts the assumption on Case 2.1. Therefore, $\Omega_2^c\cap \Omega_1^+=\emptyset$. 
	
	\quad \underline{Case 2.2}. If $(t,x,y)\in \Omega_2^c\cap \Omega_1^c$, then
	\begin{align*}
		y=Z_i(x,y)+\varphi_i(x-Z_i(x,y)),\quad i=1,2.
	\end{align*}
By $\varphi_2(x)\leq\varphi_1(x)$, we have
	\begin{align*}
		0&=\big(Z_1(x,y)+\varphi_1(x-Z_1(x,y))\big)-\big(Z_2(x,y)+\varphi_2(x- Z_2(x,y))\big)\\
		&=Z_1(x,y)-Z_2(x,y)+\varphi_1(x-Z_1(x,y))-\varphi_2(x- Z_2(x,y))\\
		&\geq Z_1(x,y)-Z_2(x,y)+\varphi_2(x-Z_1(x,y))-\varphi_2(x- Z_2(x,y))\\
		&=Z_1(x,y)-Z_2(x,y)-\varphi'_2 (\theta)(Z_1(x,y)-Z_2(x,y))\\
		&=(Z_1(x,y)-Z_2(x,y))(1-\varphi'_2 (\theta)).
	\end{align*}
	Thus, if $(t,x,y)\in \Omega_2^c\cap \Omega_1^c$, then $Z_1(x,y)\leq Z_2(x,y)$. Therefore, we have $\di u^R_1(t,x,y)=\frac{Z_1(x,y)}{t}\leq \frac{Z_2(x,y)}{t}=u^R_2(t,x,y).$

\quad \underline{Case 2.3}. If $(t,x,y)\in \Omega_2^c\cap \Omega_1^-$, then $u^R_1(t,x,y)=u_-\leq u^R_2(t,x,y)$.\\

\
	
	\underline{Case 3}. For the case $(t,x,y)\in \Omega_2^-$:
	
	\quad \underline{Case 3.1}. If $(t,x,y)\in \Omega_2^-\cap \Omega_1^-$, then $u^R_1(t,x,y)=u^R_2(t,x,y)=u_-$.
	
	\quad \underline{Case 3.2}.  If $(t,x,y)\in \Omega_2^-\cap \Omega_1^c$, then
	\begin{align}
		y=Z_1(x,y)+\varphi_1(x-Z_1(x,y)).
	\end{align}
By $\varphi_2(x)\leq\varphi_1(x)$, we have
	\begin{align*}
		y-u_-t-\varphi_2(x-u_-t)&=Z_1(x,y)-u_-t+\varphi_1 (x-Z_1(x,y))-\varphi_2(x-u_-t)\\
		&\geq Z_1(x,y)-u_-t+\varphi_2(x-Z_1(x,y))-\varphi_2(x-u_-t)\\
		&=Z_1(x,y)-u_-t-\varphi'_2 (\theta)(Z_1(x,y)-u_-t)\\
		&=(Z_1(x,y)-u_-t)(1-\varphi'_2 (\theta))\geq 0.
	\end{align*}
	Then $(t,x,y)\notin \Omega_2^-$, this contradicts the assumption in Case 3.2. Therefore, $\Omega_2^-\cap \Omega_1^c=\emptyset$.
	
	\quad \underline{Case 3.3}.  If $(t,x,y)\in \Omega_2^-\cap \Omega_1^+$, then we can similarly get $\Omega_2^-\cap \Omega_1^+=\emptyset$.
	
	Collecting all the above considerations, \cref{RRleq} is proved.
\end{proof}

\begin{Lem}\label{RRM}
	Suppose that $\varphi_2(x)=\varphi_1(x)+M$ with $\varphi_1(x)$ satisfying \cref{2DH} and  $M$ being a constant, then for $\forall t>0$,
	\begin{align}\label{uur}
		\norm{u^R_1(t,x,y)-u^R_2(t,x,y)}_{L^{\infty}(\mathbb R^2)}\leq \frac{C}{t},
	\end{align}
	where $C$ is a positive constant depending on $M$ and $d_0$, but independent of $t>0$.
\end{Lem}

\begin{proof}
	Without loss of generality, we consider the case $M>0$. Recalling the definition of the six regions $\Omega_i^-, \Omega_i^c, \Omega_i^+$ in \eqref{region} for $i=1,2$.  Here we only concern the case  $\varphi_2(x)=\varphi_1(x)+M$ with  $M>0$. 
Still we have  $\mathbb R_+\times\mathbb R^2=\Omega_i^-\cup\Omega_i^c\cup\Omega_i^+$ for $i=1,2$.

Now we prove \cref{RRM} by the following three cases:\\
	
\underline{Case 1}. If $(t,x,y)\in \Omega_1^-$, then
	by the definition of $\Omega_2^-$, we have $\Omega_1^-\subset\Omega_2^-$. Thus $u^R_1(t,x,y)=u^R_2(t,x,y)\equiv u_-$, and then \eqref{uur} is obviously true.\\
	
\underline{Case 2}. If $(t,x,y)\in \Omega_1^c$:
	
		\quad \underline{Case 2.1}.  If $(t,x,y)\in \Omega_1^c\cap \Omega_2^-$, then 
	\begin{align}
		y-Z_1(x,y)-\varphi_1(x-Z_1(x,y))=0\quad\text{and} \quad y-u_-t-M-\varphi _1(x-u_-t)<0.
	\end{align}
	Thus
	\begin{align*}
		0&>Z_1(x,y)-u_-t+\varphi_1(x-Z_1(x,y))-\varphi_1(x-u_-t)-M\\
		&=Z_1(x,y)-u_-t-\varphi'_1 (\theta)(Z_1(x,y)-u_-t)-M\\
		&=(Z_1(x,y)-u_-t)(1-\varphi'_1 (\theta))-M\\
		&\geq(Z_1(x,y)-u_-t)d_0-M.
	\end{align*}
	Then $\di |u^R_1(t,x,y)-u^R_2(t,x,y)|=\frac{Z_1(x,y)}{t}-u_-\leq \frac{M}{d_0t}$. Therefore, \eqref{uur} is proved.
	
	\quad \underline{Case 2.2}.  If $(t,x,y)\in \Omega_1^c\cap \Omega_2^c$, by
	\begin{align*}
		y-Z_1(x,y)-\varphi_1 (x-Z_1(x,y))=0\quad\text{and} \quad y- Z_2(x,y)-M-\varphi_1 (x-Z_2(x,y))=0,
	\end{align*}
we have
	\begin{align*}
		0&=Z_1(x,y)-Z_2(x,y)+\varphi_1(x-Z_1(x,y))-\varphi_1(x- Z_2(x,y))-M\\
		&=Z_1(x,y)- Z_2(x,y)-\varphi'_1 (\theta)(Z_1(x,y)-Z_2(x,y))-M\\
		&=(Z_1(x,y)- Z_2(x,y))(1-\varphi'_1 (\theta))-M.
	\end{align*}
	Then $$|u^R_1(t,x,y)-u^R_2(t,x,y)|=\left|\frac{Z_1(x,y)- Z_2(x,y)}{t}\right|=\frac{M}{t(1-\varphi'_1 (\theta))}\leq \frac{M}{d_0t}.$$ Therefore, \eqref{uur} is proved.
	
	\quad \underline{Case 2.3}.  If $(t,x,y)\in \Omega_1^c\cap \Omega_2^+,$ then similar to Case 3.2 in the proof of \cref{RRleq}, we have $\Omega_1^c\cap \Omega_2^+=\emptyset$.\\
	
\underline{Case 3}. If $(t,x,y)\in \Omega_1^+$:
	
\quad \underline{Case 3.1}.  If $(t,x,y)\in \Omega_1^+\cap \Omega_2^-$, then
	\begin{align*}
		y-u_+t-\varphi_1(x-u_+t)>0\quad\text{and} \quad  y-u_-t-M-\varphi_1(x-u_-t)<0.
	\end{align*}
	Thus we have 
	\begin{align*}
		0>y-u_-t-M-\varphi (x-u_-t)&>u_+t-u_-t+\varphi_1 (x-u_+t)-\varphi_1 (x-u_-t)-M\\
		&=(u_+-u_-)t-\varphi'_1 (\theta)(u_+-u_-)t-M\\
		&=(u_+-u_-)(1-\varphi'_1 (\theta))t-M\\
		&\geq(u_+-u_-)d_0t-M,
	\end{align*}
which yields that if $t>\frac{M}{d_0(u_+-u_-)}$, then  $y-u_-t-M-\varphi (x-u_-t)>0$,  this implies an contradiction.
 Therefore, for $t>\frac{M}{d_0(u_+-u_-)}$, $\Omega_1^+\cap \Omega_2^-=\emptyset$. On the other hand, for $0<t<\frac{M}{d_0(u_+-u_-)}$, we have $$|u^R_1(t,x,y)-u^R_2(t,x,y)|\leq u_+-u_-\leq\frac{M}{d_0t}.$$ Therefore, \eqref{uur} is proved.

	\quad \underline{Case 3.2}.  If $(t,x,y)\in \Omega_1^+\cap \Omega_2^c$, then
	\begin{align*}
		y-u_+t-\varphi_1(x-u_+t)>0\quad\text{and} \quad  y-Z_2(x,y)-M-\varphi_1(x-Z_2(x,y))=0.
	\end{align*}
	Similarly, we have
	\begin{align*}
		0&<Z_2(x,y)-u_+t+\varphi_1(x-Z_2(x,y))-\varphi_1(x-u_+t)+M\\
		&=Z_2(x,y)-u_+t-\varphi'_1 (\theta)(Z_2(x,y)-u_+t)+M\\
		&=(Z_2(x,y)-u_+t)(1-\varphi'_1 (\theta))+M.
	\end{align*}
	Then $|u^R_1(t,x,y)-u^R_2(t,x,y)|=u_+-\frac{Z_2(x,y)}{t}<\frac{M}{(1-\varphi'_1 (\theta))t}\leq \frac{M}{d_0t}$. Therefore, \eqref{uur} is proved.
	
	\quad \underline{Case 3.3}.  If  $(t,x,y)\in \Omega_1^+\cap \Omega_2^+$, we have $u^R_1(t,x,y)=u^R_2(t,x,y)=u_+$, and then  \eqref{uur} is obviously true.
	
The proof of  \eqref{uur} is completed.
\end{proof}


Now we can  prove \cref{prop12}.

\begin{proof}
	 Since $\norm{\varphi_1(x)-\varphi_2(x)}_{L^{\infty}(\mathbb R)}<+\infty,$ and $\varphi_i(x)\in C^2(\mathbb R)(i=1,2)$, so there exists a constant $M>0$ such that \begin{align}\label{lll}
		\varphi_2(x)-M\leq  \varphi_1(x)\leq\varphi_2(x)+M.\end{align}
	
	Let  $u^R_{2\pm}(t,x,y)$ be the rarefaction wave solution to the Riemann problem \eqref{eq0}-\eqref{0data} corresponding to the initial discontinuity $y=\varphi_2(x)\pm M$. According to \cref{RRleq} and \eqref{lll}, we have
	\begin{align}\label{lllr}
		u^R_{2+}(t,x,y)\leq  u^R_1(t,x,y)\leq u^R_{2-}(t,x,y).\end{align}
	By \cref{RRM} and \eqref{lllr}, we have
	\begin{align*}
		&\norm{u^R_1(t,x,y)-u^R_2(t,x,y)}_{L^{\infty}(\mathbb R^2)}\\\leq& \norm{u^R_1(t,x,y)-u^R_{2-}(t,x,y)}_{L^{\infty}(\mathbb R^2)}+\norm{u^R_{2-}(t,x,y)-u^R_2(t,x,y)}_{L^{\infty}(\mathbb R^2)}\\\leq&
		\norm{u^R_{2+}(t,x,y)-u^R_{2-}(t,x,y)}_{L^{\infty}(\mathbb R^2)}+\norm{u^R_{2-}(t,x,y)-u^R_2(t,x,y)}_{L^{\infty}(\mathbb R^2)}\\\leq&
		\frac{C}{t}.
	\end{align*}
Thus we complete the proof of \cref{prop12}. 
\end{proof}
Next we prove \cref{Z12}.
\begin{proof}
	From \eqref{zi}, $\forall (x, y)\in \mathbb{R}^2,$ we have
	\begin{equation}\label{z12v12} Z_1(x,y)-Z_2(x,y)+\varphi_1 (x-Z_1)-\varphi_2 (x-Z_2)=0.\end{equation}
	Thus
	\begin{align*}
		0=&Z_1(x,y)-Z_2(x,y)+\varphi_1 (x-Z_1)-\varphi_1 (x-Z_2)+\varphi_1 (x-Z_2)-\varphi_2(x-Z_2)\\ =&(Z_1(x,y)-Z_2(x,y))(1-\varphi'_1 (\theta))+\varphi_1 (x-Z_2)-\varphi_2 (x-Z_2), 
	\end{align*}
	for some $\theta$. Then by \eqref{2DH}, we have
	\begin{equation*}
		\left|Z_1(x,y)-Z_2(x,y)\right| \leq d_0^{-1}\left| \varphi_1 (x-Z_2)-\varphi_2 (x-Z_2)\right|.
	\end{equation*}
	Since $w_0'(z)>0$, thus it holds that
	\begin{equation}
		\begin{aligned}
		\left|w_0(Z_1(x,y))-w_0(Z_2(x,y))\right|&= w_0'(\theta)\left|Z_1(x,y)-Z_2(x,y)\right| \label{wz12}\\
		&\leq (w_0'(Z_1(x,y))+w_0'(Z_2(x,y)))	\left|Z_1(x,y)-Z_2(x,y)\right| \\
		&\leq d_0^{-1}w_0'(Z_1(x,y))\left| \varphi_1 (x-Z_2)-\varphi_2 (x-Z_2)\right|\\&~~~~~+d_0^{-1}w_0'(Z_2(x,y))\left| \varphi_1 (x-Z_2)-\varphi_2 (x-Z_2)\right|\\
		&:=J_1+J_2,
	\end{aligned}
	\end{equation}
and then
	\begin{equation}
		\norm{w_0(Z_1(x,y))-w_0(Z_2(x,y))}_{L^2(\mathbb R^2)}\leq  \norm{J_1}_{L^2(\mathbb R^2)}+\norm{J_2}_{L^2(\mathbb R^2)}.
	\end{equation}
	By \cref{LZ} (i) and \eqref{2DH}, we have the Jacobi determinant
	$$\begin{array}{ll}
		\di \frac{\partial (Z_1, x-Z_2)}{\partial(x,y)}&\di =\begin{vmatrix}
			\partial_xZ_1(x,y)&\partial_yZ_1(x,y)\\\partial_x(x-Z_2(x,y))&\partial_y(x-Z_2(x,y))
		\end{vmatrix}=\begin{vmatrix}
			\partial_xZ_1(x,y)&\partial_yZ_1(x,y)\\1-\partial_xZ_2(x,y)&-\partial_yZ_2(x,y)
		\end{vmatrix}
		\\[5mm]
		&=
		\begin{vmatrix}
			\partial_xZ_1(x,y)&1-\partial_xZ_1(x,y)\\\partial_yZ_2(x,y)&-\partial_yZ_2(x,y)
		\end{vmatrix}
		=-\partial_yZ_2(x,y)		<0.
	\end{array}$$
Due to the fact $0<\partial_yZ_2(x,y)<d_0^{-1}$, we have
\begin{equation}
		\begin{aligned}
		\norm{J_1}_{L^2(\mathbb R^2)}^2&=\iint_{\mathbb R^2}d_0^{-2}w_0'^2(Z_1(x,y))\left| \varphi_1 (x-Z_2)-\varphi_2 (x-Z_2)\right|^2dxdy\\
		&\leq C\iint_{\mathbb R^2}w_0'^2(\alpha))\left| \varphi_1 (\beta)-\varphi_2 (\beta)\right|^2d\alpha d\beta\\
		&=C\int_{\mathbb R}\left| \varphi_1 (\beta)-\varphi_2 (\beta)\right|^2 d\beta<+\infty\label{w0jl2}.
	\end{aligned}
\end{equation}
Similarly, we can obtain $\norm{J_2}_{L^2(\mathbb R^2)}^2<+\infty$. Therefore, it holds that
$$
\norm{w_0(Z_1(x,y))-w_0(Z_2(x,y))}_{L^2(\mathbb R^2)}< +\infty.
$$
Furthermore, we have
	\begin{align*}
		&\left|\partial _x(w_0(Z_1(x,y))-w_0(Z_2(x,y)))\right|\\=&\left| w_0'(Z_1)(\partial_xZ_1-\partial_xZ_2)+(w_0'(Z_1)-w_0'(Z_2))\partial_xZ_2\right|\\\leq & w_0'(Z_1)\left|\frac{\varphi_1'(x-Z_1)-\varphi_2'(x-Z_2)}{(1-\varphi_1'(x-Z_1))(1-\varphi_2'(x-Z_2))}\right|+C\left| w_0''(\theta)(Z_1-Z_2)\right| \\
		\leq &C w_0'(Z_1)\left|\varphi_1'(x-Z_1)-\varphi_2'(x-Z_2)\right|+Cw_0'(\theta)\left|Z_1(x, y)-Z_2(x, y)\right| \\
		\leq& Cw_0'(Z_1)|\varphi_1'(x-Z_1)-\varphi_2'(x-Z_1)|+Cw_0'(Z_1)|\varphi_2'(x-Z_1)-\varphi_2'(x-Z_2)|\\&+Cw_0'(\theta)\left|Z_1(x,y)-Z_2(x,y)\right|\\
		=&Cw_0'(Z_1)|\varphi_1'(x-Z_1)-\varphi_2'(x-Z_1)|+ Cw_0'(Z_1)\varphi_2''(\theta_1)|Z_1(x,y)-Z_2(x,y)|\\&+Cw_0'(\theta)\left|Z_1(x,y)-Z_2(x,y)\right|,
	\end{align*}
	where we have used $\eqref{2DH}$ and $ |w_0''(\theta)|\leq Cw_0'(\theta)$.
	Then by \eqref{idh1} and the similar process as in deriving \eqref{wz12} to \eqref{w0jl2}, we can prove 
	$$\norm{\partial _x(w_0(Z_1(x,y))-w_0(Z_2(x,y)))}_{L^2(\mathbb R^2)}< +\infty.$$ Similarly, we can get $$\norm{\partial _y(w_0(Z_1(x,y))-w_0(Z_2(x,y)))}_{L^2(\mathbb R^2)}< +\infty.$$
	
Therefore, we have
	$\norm{w_0(Z_1(x,y))-w_0(Z_2(x,y))}_{H^1(\mathbb R^2)}< +\infty.$
	The proof of \cref{Z12} is completed.
\end{proof}

\end{document}